\documentclass{preprint-raic}
\usepackage{enumerate}
\usepackage{mathrsfs}
\usepackage{float}

\AUTHORS{
  Martin~Rai\v{c}%
    \footnote{%
      University of Ljubljana,
      University of Primorska, and
      Institute of Mathematics, Physics and Mechanics,
      Slovenia.
      \EMAIL{martin.raic@fmf.uni-lj.si}%
   }
}
\TITLE{%
  A MULTIVARIATE CENTRAL LIMIT THEOREM FOR LIPSCHITZ AND SMOOTH TEST FUNCTIONS%
} % \thanks is optional. Insert line breaks with \\

\SHORTTITLE{A multivariate CLT} 

 % Separate items with ;

\AMSSUBJ{60F05} % Edit. Separate items with ;
\ABSTRACT{%
  We provide an abstract multivariate central limit theorem with
  the Lindeberg type error bounded in terms of Lipschitz functions
  (Wasserstein $ 1 $-distance) or functions with bounded second or
  third derivatives. The result is proved by means of Stein's method.
  For sums of i.i.d.\ random vectors with finite third absolute moment,
  the optimal rate of convergence
  is established (that is, we eliminate the logarithmic factor in the
  case of Lipschitz test functions). We indicate how the result could
  be applied to certain other dependence structures, but do not derive
  bounds explicitly.
}

\expandafter\def\csname citep\endcsname{\cite}
\newcommand{\NN}{\mathbb{N}}
\newcommand{\RR}{\mathbb{R}}

\newcommand{\Colon}{\colon\>}
\newcommand{\eps}{\varepsilon}
\newcommand{\sth}{\mathrel{;}}
\newcommand{\dl}{\mathrm{d}}
\newcommand{\mathmatr}[1]{\mathbf{#1}}
\newcommand{\Id}{\mathmatr{I}}
\newcommand{\scalp}[2]{\langle #1 \, , \, #2 \rangle}
\newcommand{\Trans}{T}

\newcommand{\Scalp}[3][autosizedelim]{\csname #1l\endcsname\langle #2 \, , \, #3 \csname #1r\endcsname\rangle}

\let\scrpt=\mathscr

\DeclareMathOperator{\Pp}{\mathbb{P}}
\DeclareMathOperator{\Ee}{\mathbb{E}}

\DeclareMathOperator{\Var}{Var}
\newcommand{\Normal}{\mathcal{N}}
\DeclareMathOperator{\One}{\mathds 1}

\newcommand{\Stein}{\scrpt S}
\newcommand{\Slep}{\scrpt U}

\newcounter{component}
\newcommand{\labelcomponent}[1]{{\normalfont (C#1)}}
\newcommand{\brevelabelcomponent}[1]{$(\breve{\mbox{\normalfont C}}#1)$}
\renewcommand{\thecomponent}{\labelcomponent{\arabic{component}}}
\newcommand{\component}{\refstepcounter{component}\item[\thecomponent]}
\newcommand{\assmptn}[1]{\refitem{{\normalfont (#1)}}}

\newenvironment{compprop}{%
  \list{\labelwidth 3em labelsep 0.3em \itemindent 0pt \topsep 1ex \parsep 0pt \itemsep 1ex}{}
}{%
  \endlist
}

\begin{document}

\section{Introduction}
\label{sc:Intr}

It is well-known that, roughly speaking, a sum of many random variables with sufficiently
nice distributions, which do not differ too much in scale and are not too dependent, approximately
follows a normal distribution. This fact is referred to as the central limit theorem and has
been formulated in numerous variants. It can be readily extended to $ \RR^d $-valued random
vectors.

One of the ways to make the sloppy statement above more precise is to provide a bound on
the error in the normal approximation.  One of the ways to measure the error is to
consider expectations of test functions from a given class $ \scrpt F $: for a given
$ \RR^d $-valued random vector $ W $ and a $ d $-variate normal vector $ Z $,
consider the supremum
\begin{equation}
\label{eq:supF}
 \sup_{f \in \scrpt F} \bigl| \Ee \bigl[ f(W) \bigr] - \Ee \bigl[ f(Z) \bigr] \bigr| \, .
\end{equation}
Several classes of test functions have been taken into consideration. In many cases,
the error has been estimated optimally up to a constant. In particular, for properly
scaled partial sums of a sequence of independent and identically distributed random
vectors with finite third absolute moment, the optimal rate of convergence to the
standard normal distribution is typically $ n^{-1/2} $, where $ n $ is the number of
the summands. This rate has been established for many classes of test functions.
For indicators of convex sets, see, e.~g., Bentkus~\citep{BLja:E}, G\"otze~\citep{Gtz}
or the author's previous work \citep{CMBE}.

Another important example is the class of sufficiently smooth functions with properly%
\break
bounded partial derivatives of a given order. Most of the results in the multivariate case
are derived for classes based on the second or higher derivatives: see, for example, Goldstein
and Rinott~\citep{GRn}, Rinott and Reinert~\cite{zer}, Chatterjee and Meckes~\citep{ChatMeck},
and Reinert and R\"ollin~\citep{RntRll09}. For i.~i.~d.\ random vectors with finite third
absolute moments, the optimal rate of convergence of $ n^{-1/2} $ has been established, too.

Surprisingly, classes based on the first-order derivatives seem to be more difficult.
Typically, one simply considers the class of functions with the Lipschitz constant
bounded from above by $ 1 $; for this class, the underlying supremum \eqref{eq:supF}
is referred to as the \emph{1-Wasserstein distance} (or also Kantorovich--Rubinstein
distance). In the context of the central limit theorem,
this distance has been well-established only in the univariate
case, where again the optimal rate of convergence of $ n^{-1/2} $ has been derived
for i.~i.~d.\ random vectors with finite third absolute moments: it can be, for
example, deduced from Theorem~1 of Barbour, Karo\'nski and Ruci\'nski~\citep{BKR}.
To the best of the author's knowledge, this is not the case in higher dimensions.
In this case, a suboptimal rate of $ n^{-1/2} \log n $ has been derived by
Galouet, Mijoule and Swan~\citep{GaMiSwW} as well as by Fang, Shao and Xu~\citep{FngShaoXuW}
(see also the references therein).

In the present paper, we succeed to remove the logarithmic factor -- the latter only
remains in the dependence on the dimension: see the bound \eqref{eq:Ind:CLT1}.
In addition, \eqref{eq:Ind:CLT1} does not require finiteness of the third
absolute moments -- it is a Lindeberg type bound.

The result is derived by Stein's method, which has been introduced in
\citep{St0}. The main idea of the method is to reduce the estimation of
the error to the estimation of expectations related to a solution of
a differential equation, which is now called Stein equation. For Lipschitz
test functions, the third derivatives of the solution apparently play the
key role. Unfortunately, they cannot be properly bounded -- see a
counterexample in the author's previous paper \citep{LFDec}, Remark~2
ibidem. However, we show that the third derivatives can be circumvented by
the second and fourth ones, which behave properly. The key step is carried out in
the estimate~\eqref{eq:circum}.

One of its major advantages of Stein's method is that it is by no means limited
to sums of independent random vectors. It works well under various dependence structures:
for an overview, the reader is referred to Barbour and Chen \citep{BCI,BCA}. Moreover, the random
variable to be approximated need not be a sum. We point out two approaches
called size and zero biassing: see Baldi, Rinott and Stein~\citep{BRS},
respectively Goldstein and Reinert~\citep{zer}.

In the present paper, we introduce a structure which generalizes the size and
zero biassing, and state two abstract results, Theorems~\ref{th:CLT:size} and
\ref{th:CLT:zero}. In Section~\ref{sc:Ind}, the latter is applied to sums
of independent random vectors, but we indicate how the abstract results
can be used beyond independence. In particular, our approach may be applied
in future for sums of random vectors where we can efficiently compare
conditional distributions given particular summands with their unconditional
counterparts -- see Example~\ref{ex:sum}. In addition, we indicate how it could
be used for Palm processes -- see Example~\ref{ex:PP}. However, we do not derive
explicit bounds for either of these two cases.

\section{Notation, assumptions and general results}
\label{sc:Res}

First, we introduce some basic notation:
\begin{itemize}
\item $ \Id_d $ denotes the $ d \times d $ identity matrix.
\item $ |\cdot| $ denotes the Euclidean norm.
\item For $ F \Colon \RR^d \to V $, where $ V $ is a finite-dimensional vector space,
define $ \Normal F := \Ee \bigl[ F(Z) \bigr] $, where $ Z $
is a standard $ d $-variate normal vector.
\item For $ f \Colon \RR^d \to \RR $ and $ r \in \NN $, denote by
$ \nabla^r f(w) $ the $ r $-th derivative of $ f $ at $ w $. This is a $ r $-fold tensor;
see Subsection~\ref{ssc:Der} for more detail.
\item Denote by $ M_0(f) $ the supremum norm of a function $ f $. Furthermore,
for an $ (r-1) $-times differentiable function $ f \Colon \RR^d \to \RR $, define
\begin{equation}
\label{eq:Mr}
 M_r(f) := \sup_{\substack{x, y \in \RR^d\\ x \ne y}} \frac{\bigl| \nabla^{r-1} f(x) - \nabla^{r-1} f(y) \bigr|_\vee}{|x - y|}
 \, ,
\end{equation}
where $ |\cdot|_\vee $ denotes the injective norm: see Subsections~\ref{ssc:Ten} and \ref{ssc:Der}.
If $ f $ is not everywhere $ (r-1) $-times differentiable, we put $ M_r(f) := \infty $.
\end{itemize}
For more details on notation and definitions, see Section~\ref{sc:Prel}.

\begin{remark}
\label{rk:Rademacher}
This way, if $ M_r(f) < \infty $, then $ \nabla^{r-1} f $ exists everywhere and is Lip\-schitz.
In this case, $ \bigl| \nabla^r f(x) \bigr|_\vee \le M_r(f) $ for all $ x $ where $ \nabla^{r-1} f $
is differentiable.
% \begin{tocheck}
% In this case, by Rademacher's theorem (see Federer~\citep{Fed}, Theorem~3.1.6),
% % Section~3.1.2 of Evans and Gariepy~\citep{EvG}
% $ \nabla^{r-1} f $ is almost everywhere differentiable. In addition,
% $ M_r(f) = \sup_x \bigl| \nabla^r f(x) \bigr|_\vee $,
% where the supremum runs over all points where $ \nabla^{r-1} f $ is differentiable.
% \end{tocheck}
\end{remark}

\begin{remark}
\label{rk:MBd}
If $ M_r(f) < \infty $, there exist constants $ C $ and $ D $, such that $ |f(x)| \le C + D \, |x|^r $
for all $ x \in \RR^d $.
\end{remark}

Our main main results, Theorem~\ref{th:CLT:size} and \ref{th:CLT:zero}, will be based on
various \emph{assumptions} which refer to various \emph{components}. They are listed below and
labelled. Components are labelled by the letter C and a number, other labels stand for properties.
% Although Theorem~\ref{th:LipCLT} does not need all listed items, the rest can
% serve as alternatives.

\begin{compprop}

\component\label{comp:W}
Consider a $ \RR^d $-valued random vector $ W $.

\assmptn{St}\label{ass:St}
Refers to \ref{comp:W}.
\textit{Suppose that $ \Ee |W|^2 < \infty $, $ \Ee W = 0 $ and $ \Var(W) = \Id_d $.}

\component\label{comp:xi}
Consider a measurable space $ (\Xi, \scrpt X) $.

\component\label{comp:V}
For each $ \xi \in \Xi $, consider a $ \RR^d $-valued random vector $ V_\xi $.
The random vectors $ V_\xi $ may be defined on different probability spaces.
Denote the underlying probability measures and expectations by $ \Pp_\xi $ and
$ \Ee_\xi $. Assume that the maps $ \xi \mapsto \Pp_\xi(V_\xi \in A) $ are
measurable for all Borel sets $ A \subseteq \RR^d $.

\component\label{comp:mu}
\xdef\compMucounter{\arabic{component}}
Consider an $ \RR^d $-valued measure $ \mu $ on $ (\Xi, \scrpt X) $.

\assmptn{S}\label{ass:size}
Refers to \ref{comp:W}--\ref{comp:mu}.
\textit{Suppose that $ \Ee |W| < \infty $ and that
\begin{equation}
\label{eq:size}
 \Ee \bigl[ f(W) W \bigr]
 =
 \int_{\Xi} \Ee_\xi \bigl[ f(V_\xi) \bigr] \, \mu(\dl \xi)
\end{equation}
for all bounded measurable functions $ f \Colon \RR^d \to \RR $.}

\end{compprop}

\begin{remark}
Assumption~\ref{ass:size} can be regarded as a generalization of the \emph{size-biassed
transformation} (see Baldi, Rinott and Stein~\citep{BRS}): for $ d = 1 $ and $ W \ge 0 $,
\ref{ass:size} is fulfilled with $ \Xi = \{ 0 \} $ and $ \mu(\{ 0 \}) = \Ee W $
if and only if the distribution of $ V_0 $ is the size-biassed distribution of $ W $.
\end{remark}

\begin{remark}
\label{rk:sizeE}
Under~\ref{ass:size}, we have $ \mu(\Xi) = \Ee W $ (put $ f(w) = 1 $).
\end{remark}

\noindent
We list two more important examples where Assumption~\ref{ass:size} is satisfied.

\begin{example}
\label{ex:sum}
Let $ \scrpt I $ be a countable set and let $ W = \sum_{i \in \scrpt I} X_i $ --
suppose that the latter sum exists almost surely.
Define $ \Xi := \scrpt I \times \RR^d $ and let $ \scrpt X $ be the product
(in terms of $ \sigma $-algebras) of the power set of $ \scrpt I $ and the Borel
$ \sigma $-algebra on $ \RR $. Choose probability measures $ \Pp_{i,x} $, $ i \in \scrpt I $,
$ x \in \RR^d $, so that for each $ i \in \scrpt I $, they determine the conditional distribution
of $ W $ given $ X_i $, i.~e., $ \Pp_{i,x}(A) = \Pp(W \in A \mid X_i = x) $.
For all $ i $ and $ x $, let $ V_{i,x} $ be the identity on $ \RR^d $.
Letting $ \mu(\{ i \} \times A) = \Ee \bigl[ X_i \One(X_i \in A) \bigr] $, we find that
\begin{align*}
 \Ee \bigl[ f(W) W \bigr]
 &=
 \sum_{i \in \scrpt I} \Ee \bigl[ f(W) X_i \bigr] \\
 &=
 \sum_{i \in \scrpt I} \int_{\RR^d} f(w) \, \Pp_{i,x}(\dl w) \, x \, \scrpt L(X_i)(\dl x) \\
 &=
 \int_{\scrpt I \times \RR^d} \bigl( \Ee_{(i, x)} f \bigr) \, \mu(\dl i \otimes \dl x)
 \, ,
\end{align*}
implying \ref{ass:size}.
\end{example}

\begin{example}
\label{ex:PP}
Let $ \mathcal P $ be a random point process on a space $ \Xi $ admitting
\emph{Palm processes} $ \mathcal P_\xi $, $ \xi \in \Xi $. Intuitively,
$ \mathcal P_\xi $ is the conditional distribution of $ \mathcal P $ given that
there is a point at $ \xi $. Strictly speaking, the Palm processes are characterized
by the formula
\[
 \Ee \left[ \int_{\Xi} \Phi(\xi, \mathcal P) \, \mathcal P(\dl \xi) \right]
 =
 \int_{\Xi} \Ee_\xi \Phi(\xi, \mathcal P_\xi) \, m(\dl \xi)
 \, ,
\]
where $ m $ is the mean measure of $ \mathcal P $ (for details, see Proposition~13.1.IV
of Daley and Vere-Jones~\citep{DV2}). Now take a function
$ F \Colon \Xi \to \RR^d $ and define $ W := \int_{\Xi} F(\xi) \, \mathcal P(\dl \xi) $.
Observe that
\begin{align*}
 \Ee \bigl[ f(W) W \bigr]
 &=
 \Ee \left[ \int_{\Xi} f \left( \int_{\Xi} F(\eta) \, \mathcal P(\dl \eta) \right) F(\xi) \, \mathcal P(\dl \xi) \right] \\
 &=
 \int_{\Xi} \Ee_\xi \left[ f \left( \int_{\Xi} F(\eta) \, \mathcal P_\xi(\dl \eta) \right) \right] F(\xi) \, m(\dl \xi)
 \, .
\end{align*}
Thus, we can set $ V_\xi = \int_{\Xi} F(\eta) \, \mathcal P_\xi(\dl \eta) $ and $ \mu = F \cdot m $.
\end{example}

\noindent
We continue listing components and properties required for Theorems~\ref{th:CLT:size} and \ref{th:CLT:zero}.

\begin{compprop}

\component\label{comp:nu}
For each $ \xi \in \Xi $, consider a $ \RR^d $-valued measure $ \nu_\xi $
and assume that the maps $ \xi \mapsto \nu_\xi(A) $ are measurable for all
$ A \in \scrpt X $.

\assmptn{Px}\label{ass:prox}
Refers to \ref{comp:W}--\ref{comp:V} and \ref{comp:nu}.
\textit{Suppose that $ \Ee |W| < \infty $. For each $ \xi \in \Xi $, suppose that
$ \Ee_\xi |V_\xi| < \infty $ and
\begin{equation}
\label{eq:prox}
 \Ee \bigl[ f(W) \bigr] - \Ee_\xi \bigl[ f(V_\xi) \bigr] = \int_\Xi
 \Scalp[big]{\Ee_\eta \bigl[ \nabla f(V_\eta) \bigr]}{\nu_\xi(\dl \eta)}
\end{equation}
for all continuously differentiable functions $ f \Colon \RR^d \to \RR $ with bounded derivative
(observe that $ \Ee \bigl| f(W) \bigr| $ and $ \Ee_\xi \bigl| f(V_\xi) \bigr| $ are finite
by Remark~\ref{rk:MBd}). The integral in \eqref{eq:prox} is defined according to Definition~\ref{df:IntTen}.}

\end{compprop}

\begin{remark}
Under Assumption~\ref{ass:prox}, we can measure the proximity of the distribution of $ V_\xi $
to the distribution of $ W $ in terms of $ \nu_\xi $. In particular, the 1-Wasserstein
distance is straightforward to estimate, as
$ \bigl| \Ee[f(W)] - \Ee_\xi[f(V_\xi)] \bigr| \le M_1(f) \, |\nu_\xi|(\Xi) $. However,
the latter distance will not be the only measure of proximity we shall need.
\end{remark}

\begin{example}
\label{ex:NewtLeib}
If $ W $ and $ V_\xi $ are defined on the same probability space, observe that
\begin{equation}
\label{eq:NewtLeib}
 \Ee[f(W)] - \Ee[f(V_\xi)]
 =
 \int_0^1 \Ee \bigl[ \Scalp{\nabla f \bigl( (1 - t) V_{\xi} + t W \bigr)}{W - V_\xi} \bigr] \,\dl t
 \, .
\end{equation}
Now suppose that there is a measurable map $ \psi \Colon \Xi \times [0, 1] \times \RR^d \to \Xi $, such that
the (unconditional) distribution of $ V_{\psi(\xi, t, y)} $ agrees with the conditional distribution
of $ (1 - t) V_{\xi} + t W $ given $ W - V_\xi = y $. Then \eqref{eq:NewtLeib} can be rewritten
as
\begin{equation}
\label{eq:NewtLeibCond}
 \Ee[f(W)] - \Ee[f(V_\xi)]
 =
 \int_0^1 \int_{\RR^d} \Scalp{\Ee_{\psi(\xi, t, y)}[\nabla f(V_{\psi(\xi, t, y)})]}{y} \, \scrpt L(W - V_\xi)(\dl y) \,\dl t
 \, .
\end{equation}
Now put
\begin{align*}
 \nu_\xi(B)
 &:=
 \int_0^1 \int_{\RR^d} \One \bigl( \psi(\xi, t, y) \in B \bigr) \, y \, \scrpt L(W - V_\xi)(\dl y) \,\dl t \\
 &=
 \int_0^1 \Ee \Bigl[ (W - V_\xi) \One \bigl( \psi(\xi, t, W - V_\xi) \in B \bigr) \Bigr] \,\dl t \, .
\end{align*}
A standard argument shows that
\[
 \int h \,\dl \nu_\xi
 =
 \int_0^1 \int_{\RR^d} h \bigl( \psi(\xi, t, y) \bigr) \, y \, \scrpt L(W - V_\xi)(\dl y) \,\dl t
 =
 \int_0^1 \Ee \Bigl[ h \bigl( \psi(\xi, t, W - V_\xi) \bigr) (W - V_\xi) \Bigr] \,\dl t
\]
for all bounded measurable functions $ h $.
Combining with \eqref{eq:NewtLeibCond}, \eqref{eq:prox} follows.
\end{example}

% \begin{remark}
% \label{rk:prox:vec}
% In \ref{ass:prox}, $ f $ can be replaced with any suitable vector-valued map.
% \begin{todo}
% Well well, what is then the right hand side of \eqref{eq:prox}?
% \end{todo}
% \end{remark}

Before formulating the first main result, Theorem~\ref{th:CLT:size}, we introduce some more quantities:
%
%.% Changes from the last version:
%
% \beta_\xi                -->    \beta_1^{(\xi)}
% \beta'_\xi               -->    \beta_2^{(\xi)}
% \beta^*_\xi(a, b)        -->    \beta_{12}^{(\xi)}
% \beta^{**}_\xi(a, b, c)  -->    \beta_{123}^{(\xi)}(a, b, c)
%
% \beta                 -->    \beta_2
% \beta'                -->    \beta_3
% \beta^*(a, b)         -->    \beta_{23}(a, b)
% \beta^{**}(a, b, c)   -->    \beta_{234}(a, b, c)
%
% \breve \beta'                -->    \breve \beta_3
% \breve \beta^*(a, b)         -->    \breve \beta_{23}(a, b)
% \breve \beta^{**}(a, b, c)   -->    \breve \beta_{234}(a, b, c)
%
% {@eq:betasize2}   -->    {eq:betasize1-2}
% {@eq:betasize}    -->    {eq:betasize2-3}
% {@eq:betasize*}   -->    {eq:betasize12-23}
% {@eq:betasize**}  -->    {eq:betasize123-234}
%
% {@eq:betazero}    -->    {eq:betazero-3}
% {@eq:betazero*}   -->    {eq:betazero-23}
% {@eq:betazero**}  -->    {eq:betazero-234}
%
\begin{align}
\label{eq:betasize1-2}
 \beta_1^{(\xi)} &:= |\nu_\xi|(\Xi) \, ,
 &
 \beta_2 &:= \int_\Xi \beta_1^{(\xi)} \, |\mu|(\dl \xi) \, ,
\\
\label{eq:betasize2-3}
 \beta_2^{(\xi)} &:= \int_\Xi \beta_1^{(\eta)} \, |\nu_\xi|(\dl \eta) \, ,
 &
 \beta_3 &:= \int_\Xi \beta_2^{(\xi)} \, |\mu|(\dl \xi) \, ,
\\
\label{eq:betasize12-23}
 \beta_{12}^{(\xi)}(a, b) &:= \int_\Xi \min \bigl\{ a, \, b \, \beta_1^{(\eta)} \bigr\} \, |\nu_\xi|(\dl \eta) \, ,
 &
 \beta_{23}(a, b) &:= \int_\Xi \beta_{12}^{(\xi)}(a, b) \, |\mu|(\dl \xi) \, ,
\\
\label{eq:betasize123-234}
 \beta_{123}^{(\xi)}(a, b, c) &:= \int_\Xi \min \Bigl\{
   a, \, b \, \beta_1^{(\eta)} + c \sqrt{\beta_2^{(\eta)}}
 \Bigr\} \, |\nu_\xi|(\dl \eta) \, ,
 &
 \beta_{234}(a, b, c) &:= \int_\Xi \beta_{123}^{(\xi)}(a, b, c) \, |\mu|(\dl \xi) \, .
\end{align}

\begin{theorem}
\label{th:CLT:size}
Under Assumptions~\ref{ass:St}, \ref{ass:size} and \ref{ass:prox}, and $ \beta_2 < \infty $, we have
\begin{align}
 \label{eq:CLT3xi}
 \bigl| \Ee \bigl[ f(W) \bigr] - \Normal f \bigr|
 &\le
 \frac{\beta_3}{3} \, M_3(f) \, ,
\\
 \label{eq:CLT2xi}
 \bigl| \Ee \bigl[ f(W) \bigr] - \Normal f \bigr|
 &\le
 \beta_{23} \left(
   1, \, \frac{\sqrt{2 \pi}}{4}
 \right) M_2(f) \, ,
\\
 \label{eq:CLT1xi}
 \bigl| \Ee \bigl[ f(W) \bigr] - \Normal f \bigr|
 &\le
 \beta_{234} \bigl( 1.8, \, 3.58 + 0.55 \log d, \, 3.5 \bigr) \, M_1(f)
 \, .
\end{align}
More precisely, for each of the inequalities, if the underlying $ M_r(f) $ in the right hand side is finite,
then $ \Ee \bigl| f(W) \bigr| $ and $ \Normal |f| $ are also finite and the inequality
holds true (under the convention $ \infty \cdot 0 = 0 $).
\end{theorem}

\noindent
We defer the proof to Subsection~\ref{ssc:Prf}.

\medskip
Now we turn to a more direct construction with the underlying counterpart
of the preceding result. Although the additional Component~\ref{comp:Mu}
satisfying Assumption~\ref{ass:zero} can be constructed from
Component~\ref{comp:mu} satisfying Assumption~\ref{ass:size}
(see Proposition~\ref{pr:size2zero}), Theorem~\ref{th:CLT:zero} might provide
better bounds than Theorem~\ref{th:CLT:size}. Moreover, as we shall see in
Subsection~\ref{ssc:Prf}, Theorem~\ref{th:CLT:size} is actually a direct consequence of
Theorem~\ref{th:CLT:zero} and Proposition~\ref{pr:size2zero}.

\begin{compprop}

\refitem{\brevelabelcomponent{\compMucounter}}\label{comp:Mu}
Consider an $ \RR^d \otimes \RR^d $-valued measure $ \breve \mu $ on $ (\Xi, \scrpt X) $.

\assmptn{Z}\label{ass:zero}
Refers to \ref{comp:W}--\ref{comp:V} and \ref{comp:Mu}.
\textit{Suppose that $ \Ee |W|^2 < \infty $ and that
\begin{equation}
\label{eq:zero}
 \Ee \bigl[ f(W) W \bigr]
 =
 \int_\Xi \breve \mu(\dl \xi) \Ee_\xi \bigl[ \nabla f(V_\xi) \bigr]
 \, ,
\end{equation}
for all continuously differentiable functions $ f \Colon \RR^d \to \RR $ with bounded derivative,
recalling the identification of a $ 2 $-tensor $ \phi $ with a linear map $ \tilde L_\phi $
in Subsection~\ref{ssc:Ten}, as well as Definition~\ref{df:IntTen} (observe that $ \Ee \bigl| f(W) W \bigr| $
is finite by Remark~\ref{rk:MBd}).}

\end{compprop}

\begin{remark}
Assumption~\ref{ass:zero} can be regarded as a more flexible variant of the \emph{zero bias transformation}
introduced by Goldstein and Reinert~\citep{zer}. Indeed, for $ d = 1 $, Assumption~\ref{ass:zero}
is fulfilled with $ \Xi = \{ 0 \} $, $ \breve \mu(\{ 0 \}) = 1 $, provided that the distribution of
$ V_0 $ is the zero bias transform of the distribution of $ W $.
\end{remark}

\noindent
Assumption~\ref{ass:zero} can be formulated alternatively in the following way:

\begin{compprop}
\assmptn{$\mathrm{Z}'$}\label{ass:zeroalt}
Refers to \ref{comp:W}--\ref{comp:V} and \ref{comp:Mu}.
\textit{Suppose that $ \Ee |W|^2 < \infty $ and that
\begin{equation}
\label{eq:zeroalt}
 \Ee \bigl[ \scalp{F(W)}{W} \bigr]
 =
 \int_\Xi \Scalp[big]{\Ee_\xi
   \bigl[ \nabla F(V_\xi) \bigr]}{\breve \mu(\dl \xi)}
\end{equation}
for all continuously differentiable maps $ F \Colon \RR^d \to \RR^d $ with bounded derivative
(observe that $ \Ee \bigl[ |F(W)| \, |W| \bigr] $ is finite by Remark~\ref{rk:MBd}).}

\end{compprop}

\begin{proposition}
\label{pr:zeroaltequiv}
Assumptions~\ref{ass:zero} and \ref{ass:zeroalt} are equivalent.
\end{proposition}

\begin{proof}
Assumption \ref{ass:zeroalt} remains the same if we require \eqref{eq:zeroalt} only to hold
for the vector functions of form $ F(w) = f(w) \, u $, where $ f $ is continuously differentiable
with bounded derivative. Recalling
\eqref{eq:ActTensScalp}, we find that in this case, \eqref{eq:zeroalt} reduces to
\[
 \Scalp[Big]{\Ee \bigl[ f(W) \bigr] W}{u}
 =
 \int_\Xi \Scalp[Big]%
   {\breve \mu(\dl \xi)}
   {u \otimes \Ee_\xi \bigl[ \nabla f(V_\xi) \bigr]}
 =
 \int_\Xi \Scalp[Big]%
   {\breve \mu(\dl \xi) \Ee_\xi \bigl[ \nabla f(V_\xi) \bigr]}{u}
 \, .
\]
However, this is equivalent to \eqref{eq:zero}.
\end{proof}

\begin{remark}
\label{rk:zerovar}
Under~\ref{ass:zero}, we have $ \Ee W = 0 $ and $ \Var(W) = \breve \mu(\Xi) $.
The first equality follows by substituting $ f \equiv 1 $ into \eqref{eq:zero}.
Substituting $ f(w) = \scalp{w}{u} $ and recalling \eqref{eq:ActPureTens},
we find that $ \breve \mu(\Xi) u = \Ee \bigl[ \scalp{W}{u} W \bigr] = \Ee(W \otimes W) u $
for all $ u \in \RR^d $, so that $ \breve \mu(\Xi) = \Ee(W \otimes W) $.
Identifying $ 2 $-tensors with matrices, we can rewrite this as $ \breve \mu(\Xi) = \Ee(W W^\Trans)
= \Var(W) $.
\end{remark}

\begin{proposition}
\label{pr:size2zero}
Assume \ref{ass:size} and \ref{ass:prox}, recall \eqref{eq:betasize1-2} and suppose that if $ \beta_2 < \infty $.
Then Assumption~\ref{ass:zero}
% and \ref{ass:zeroalt}
is satisfied for Component~\ref{comp:Mu} defined by
$ \breve \mu(B) := - \int_\Xi \int_\Xi \One(\eta \in B) \, \mu(\dl \xi) \otimes \nu_\xi(\dl \eta) $
in view of Definition~\ref{df:Phisuccmeas}.
\end{proposition}

\begin{remark}
The finiteness of $ \beta_2 $ guarantees that the tensor-valued measure $ \breve \mu $ is well-defined.
\end{remark}

\noindent
The proof of Proposition~\ref{pr:size2zero} is deferred to Subsection~\ref{ssc:sizeproxzero}.

\medskip
\noindent
Now we are about to formulate our second main result, Theorem~\ref{th:CLT:zero}.
Before the statement, we need some more quantities,
recalling that $ |\cdot|_\wedge $ denotes the projective norm -- see Subsection~\ref{ssc:Ten}:
\begin{align}
\label{eq:betazero-3}
 \breve \beta_3 &:= \int_\Xi \beta_1^{(\xi)} \, |\breve \mu|_\wedge(\dl \xi) \, ,
\\
\label{eq:betazero-23}
 \breve \beta_{23}(a, b) &:= \int_\Xi \min \bigl\{ a, \, b \, \beta_1^{(\xi)} \bigr\} \, |\breve \mu|_\wedge(\dl \xi) \, ,
\\
\label{eq:betazero-234}
 \breve \beta_{234}(a, b, c) &:= \int_\Xi \min \Bigl\{
   a, \, b \, \beta_1^{(\xi)} + c \sqrt{\beta_2^{(\xi)}}
 \Bigr\} \, |\breve \mu|_\wedge(\dl \xi) \, .
\end{align}

\begin{theorem}
\label{th:CLT:zero}
Under Assumptions~\ref{ass:St}, \ref{ass:zero} and \ref{ass:prox}, the following inequalities hold true:
\begin{align}
 \label{eq:CLT3Xi}
 \bigl| \Ee \bigl[ f(W) \bigr] - \Normal f \bigr|
 &\le
 \frac{\breve \beta_3}{3} \, M_3(f) \, ,
\\
 \label{eq:CLT2Xi}
 \bigl| \Ee \bigl[ f(W) \bigr] - \Normal f \bigr|
 &\le
 \breve \beta_{23} \left(
   1, \, \frac{\sqrt{2 \pi}}{4}
 \right) M_2(f) \, ,
\\
 \label{eq:CLT1Xi}
 \bigl| \Ee \bigl[ f(W) \bigr] - \Normal f \bigr|
 &\le
 \breve \beta_{234} \bigl( 1.8, \, 3.58 + 0.55 \log d, \, 3.5 \bigr) \, M_1(f) \, .
\end{align}
More precisely, for each of the inequalities, if the underlying $ M_r(f) $ in the right hand side is finite,
then $ \Ee \bigl| f(W) \bigr| $ and $ \Normal |f| $ are also finite and the inequality
holds true (under the convention $ \infty \cdot 0 = 0 $).
\end{theorem}

\noindent
We defer the proof to Subsection~\ref{ssc:Prf}.

\section{Application to sums of independent random vectors}
\label{sc:Ind}

Let $ \scrpt I $ be a countable set and let $ X_i $, $ i \in \scrpt I $, be independent
$ \RR^d $-valued random vectors with $ \Ee X_i = 0 $ for all $ i \in \scrpt I $.
Suppose that $ \sum_{i \in \scrpt I} \Ee |X_i|^2 < \infty $.
Then the sum $ W := \sum_{i \in \scrpt I} X_i $ exists almost surely.
Suppose that $ \Var(W) = \Id_d $.

\subsection{Construction of measure $ \breve \mu $ satisfying \ref{ass:zeroalt}}
\label{ssc:Ind:zero}

Let $ W_i := W - X_i $ and take a countinuously differentiable map $ F \Colon \RR^d \to \RR^d $
with bounded derivative. Using independence and applying Taylor's expansion, write
\begin{align*}
 \Ee \bigl[ \scalp{F(W)}{W} \bigr]
 &=
 \sum_{i \in \scrpt I} \Ee \bigl[ \scalp{F(W_i + X_i) - F(X_i)}{X_i} \bigr] \\
 &=
 \sum_{i \in \scrpt I} \int_0^1 \Ee \bigl[ \scalp{\nabla F(W_i + t X_i)}{X_i \otimes X_i} \bigr] \,\dl t \, .
\end{align*}
Now let $ \Xi := \scrpt I \times \RR^d $ and let $ \scrpt X $ be the product
(in terms of $ \sigma $-algebras) of the power set of $ \scrpt I $ and the Borel
$ \sigma $-algebra on $ \RR^d $. Put $ V_{i,x} := W_i + x $. Then we may write
\[
 \Ee \bigl[ \scalp{F(W)}{W} \bigr]
 =
 \sum_{i \in \scrpt I} \int_0^1 \int_{\RR^d} \Scalp{\Ee \bigl[ \nabla F(V_{i,tx}) \bigr]}{x \otimes x}
 \, \scrpt L(X_i)(\dl x) \,\dl t
 \, .
\]
Letting
\[
 \breve \mu(B)
 :=
 \sum_{i \in \scrpt I} \int_0^1 \int_{\RR^d} \One \bigl( (i, tx) \in B \bigr) \, (x \otimes x)
 \, \scrpt L(X_i)(\dl x) \,\dl t \, ,
\]
a standard argument shows that
\[
 \int_\Xi h \,\dl \breve \mu
 =
 \sum_{i \in \scrpt I} \int_0^1 \int_{\RR^d} h(i, t x) \, (x \otimes x)
 \, \scrpt L(X_i)(\dl x) \,\dl t
 =
 \sum_{i \in \scrpt I} \int_0^1 \Ee \bigl[ h(i, t X_i) \, (X_i \otimes X_i) \bigr] \,\dl t
\]
for all bounded measurable functions $ h \Colon \Xi \to \RR $. This proves \ref{ass:zeroalt}.

\subsection{Construction of measures $ \nu_\xi $ satisfying \ref{ass:prox}}
\label{ssc:Ind:prox}

We use the construction from Example~\ref{ex:NewtLeib}. Observing that the conditional
distribution of $ (1 - t) V_{i,x} + t W = W_i + (1 - t) x + t X_i $ given
$ W - V_{i,x} = X_i - x = y $ agrees with the (unconditional) distribution of
$ W_i + x + t y $, we may set $ \psi(i, x, t, y) := (i, x + t y) $. Therefore,
there exist $ \RR^d $-valued vector measures $ \nu_{i,x} $, such that
\[
 \int_\Xi h \,\dl \nu_{i,x}
 =
 \int_0^1 \Ee \Bigl[ h \bigl( i, x + t(W - V_{i,x}) \bigr) (W - V_{i,x}) \Bigr] \,\dl t
 =
 \int_0^1 \Ee \Bigl[ h \bigl( i, (1 - t) x + t X_i \bigr) (X_i - x) \Bigr] \,\dl t
\]
for all bounded measurable functions $ h \Colon \Xi \to \RR $, and these measures satisfy \ref{ass:prox}.

\subsection{Estimation of $ \beta_3 $, $ \beta_{23} $ and $ \beta_{234} $}

First, observe that
\begin{align*}
 \int h \,\dl |\breve \mu|_\wedge
 &\le
 \sum_{i \in \scrpt I} \int_0^1 \Ee \bigl[ |X_i|^2 h(i, t X_i) \bigr] \,\dl t \, ,
\\
 \int h \,\dl |\nu_{i,x}|
 &\le
 \int_0^1 \Ee \Bigl[ \bigl( |X_i| + |x| \bigr) \, h \bigl( i, (1 - t) x + t X_i \bigr) \Bigr] \,\dl t
 \end{align*}
for all measurable functions $ h \Colon \Xi \to [0, \infty] $. Recalling \eqref{eq:betasize1-2},
\eqref{eq:betasize2-3} and \eqref{eq:betazero-3}--\eqref{eq:betazero-234}, we estimate
\begin{align*}
 \beta_1^{(i,x)}
 &\le
 \Ee |X_i| + |x| \, ,
\\
 \breve \beta_3
 &\le
 \sum_{i \in \scrpt I} \int_0^1 \Ee \bigl[ |X_i|^2 \beta_1^{(i, t X_i)} \bigr] \,\dl t
 \le
 \sum_{i \in \scrpt I} \int_0^1 \Ee \Bigl[ |X_i|^2 \bigl( \Ee |X_i| + t |X_i| \bigr) \Bigr] \,\dl t
 \le
 \frac{3}{2} \sum_{i \in \scrpt I} \Ee |X_i|^3
 \, ,
\end{align*}
with the last inequality being due to Jensen's inequality. Similarly,
\[
 \breve \beta_{23}(a, b)
 \le
 \sum_{i \in \scrpt I} \int_0^1 \Ee \Bigl[ |X_i|^2 \min \bigl\{ a, b \, \beta_1^{(i, t X_i)} \bigr\} \Bigr] \,\dl t
 \le
 \sum_{i \in \scrpt I}
   \Ee \biggl[ |X_i|^2 \int_0^1 \min \bigl\{ a, b \bigl( \Ee |X_i| + t |X_i| \bigr) \bigr\} \,\dl t \biggr]
 \, .
\]
Applying the inequality $ \int_0^1 \min \bigl\{ f(t), g(t) \bigr\} \,\dl t
\le \min \bigl\{ \int_0^1 f(t) \,\dl t, \int_0^1 g(t) \,\dl t \bigr\} $ and integrating,
we find that
\begin{equation}
\label{eq:Ind:beta23itm}
 \breve \beta_{23}(a, b)
 \le
 \sum_{i \in \scrpt I} \Ee \Bigl[ |X_i|^2 \min \bigl\{ a, b \, \Ee |X_i| + \tfrac{1}{2} b \, |X_i| \bigr\} \Bigr]
 \, .
\end{equation}
Finally, to bound $ \breve \beta_{234}(a, b, c) $, we first estimate
\begin{align*}
 \beta_2^{(i,x)}
 &\le
 \int_0^1 \Ee \Bigl[ \bigl( |X_i| + |x| \bigr) \beta_1^{(i, (1-t) x + t X_i)} \Bigr] \,\dl t \\
 &\le
 \int_0^1 \Ee \Bigl[ \bigl( |X_i| + |x| \bigr) \bigl( \Ee |X_i| + (1 - t) |x| + t |X_i| \bigr) \Bigr] \,\dl t \\
 &=
 \tfrac{3}{2} \Ee |X_i|^2 + 2 |x| \Ee |X_i| + \tfrac{1}{2} |x|^2 \\
 &\le
 \tfrac{3}{2} \Ee |X_i|^2 + 2 |x| \sqrt{\Ee |X_i|^2} + \tfrac{1}{2} |x|^2 \\
 &=
 \frac{1}{2} \Bigl( 3 \sqrt{\Ee |X_i|^2} + |x| \Bigr) \Bigl( \sqrt{\Ee |X_i|^2} + |x| \Bigr)
 \, .
\end{align*}
An application of the inequality between the arithmetic and the geometric mean yields
\[
 \sqrt{\beta_2^{(i,x)}}
 \le
 \tfrac{5}{4} \sqrt{\Ee |X_i|^2} + \tfrac{3}{4} \, |x|
 \, ,
\]
leading to the bound
\begin{equation}
\label{eq:Ind:beta234itm}
\begin{split}
 \breve \beta_{234}(a, b, c)
 &\le
 \sum_{i \in \scrpt I} \int_0^1 \Ee \Bigl[ |X_i|^2 \min \Bigl\{ 
   a, b \, \beta_1^{(i, t X_i)} + c \sqrt{\beta_2^{(i, t X_i)}}
 \Bigr\} \Bigr] \,\dl t \\
 &\le
 \sum_{i \in \scrpt I} \Ee \biggl[ |X_i|^2 \int_0^1 \min \Bigl\{ 
   a, \bigl( b + \tfrac{5}{4} \, c \bigr) \sqrt{\Ee |X_i|^2} + \bigl( b + \tfrac{3}{4} \, c \bigr) t \, |X_i|
 \Bigr\} \,\dl t \biggr] \\
 &\le
 \sum_{i \in \scrpt I} \Ee \biggl[ |X_i|^2 \min \Bigl\{ 
   a, \bigl( b + \tfrac{5}{4} \, c \bigr) \sqrt{\Ee |X_i|^2}
         +
      \bigl( \tfrac{1}{2} \, b + \tfrac{3}{8} \, c \bigr) |X_i|
 \Bigr\} \,\dl t \biggr]
 \, .
\end{split}
\end{equation}
For $ a, b \ge 0 $, consider functions $ h_{a,b} \Colon [0, \infty) \to [0, \infty) $
defined by
$ h_{a,b}(u) := b \, u^{3/2} $ for $ u \le \frac{a^2}{b^2} $ and
$ h_{a,b}(u) := \tfrac{3}{2} a u - \frac{a^3}{2 b^2} $ for $ u \ge \frac{a^2}{b^2} $.
Observe that $ h_{a,b} $ are convex and $ \min \bigl\{ a u, b u^{3/2} \bigr\}
\le h_{a,b}(u) \le \min \bigl\{ \tfrac{3}{2} a u, b u^{3/2} \bigr\}  $ for all $ u \ge 0 $.
Therefore, for any non-negative random variable $ X $, we have
\[
%  \min \bigl\{ a \Ee X^2, b \Ee X^2 \Ee X \bigr\}
%  \le
 \min \bigl\{ a \Ee X^2, b \, (\Ee X^2)^{3/2} \bigr\}
 \le
 h_{a,b} (\Ee X^2)
 \le
 \Ee \bigl[ h_{a,b}(X^2) \bigr] % \\
 \le
 \Ee \bigl[ \min \bigl\{ \tfrac{3}{2} a X^2, b X^3 \bigr\} \bigr]
 \, . 
\]
Further estimation of the right hand sides of \eqref{eq:Ind:beta23itm} and \eqref{eq:Ind:beta234itm}
combined with the preceding observation leads to the following Lindeberg type bounds
\begin{align*}
 \breve \beta_{23}(a, b)
 &\le
 \sum_{i \in \scrpt I} \Ee \Bigl[ |X_i|^2 \min \bigl\{ \tfrac{3}{2} \, a, b \, |X_i| \bigr\} \Bigr] +
 \sum_{i \in \scrpt I} \Ee \Bigl[ |X_i|^2 \min \bigl\{ a, \tfrac{1}{2} b \, |X_i| \bigr\} \Bigr] \\
 &\le
 \sum_{i \in \scrpt I} \Ee \Bigl[ |X_i|^2 \min \bigl\{ \tfrac{5}{2} \, a, \tfrac{3}{2} b \, |X_i| \bigr\} \Bigr] \, ,
\\
 \breve \beta_{234}(a, b, c)
 &\le
 \sum_{i \in \scrpt I} \Ee \Bigl[ |X_i|^2 \min \bigl\{
   \tfrac{3}{2} \, a, \bigl( b + \tfrac{5}{4} \, c \bigr) |X_i|
 \bigr\} \Bigr]
   +
 \sum_{i \in \scrpt I} \Ee \Bigl[ |X_i|^2 \min \bigl\{
   a, \bigl( \tfrac{1}{2} \, b + \tfrac{3}{8} \, c \bigr) |X_i|
 \bigr\} \bigr] \\
 &\le
 \sum_{i \in \scrpt I} \Ee \Bigl[ |X_i|^2 \min \bigl\{
   \tfrac{5}{2} \, a, \bigl( \tfrac{3}{2} \, b + \tfrac{13}{8} \, c \bigr) |X_i|
 \bigr\} \Bigr]
\end{align*}
and Theorem~\ref{th:CLT:zero} yields
\begin{align}
 \label{eq:Ind:CLT3}
 \bigl| \Ee \bigl[ f(W) \bigr] - \Normal f \bigr|
 &\le
 \frac{M_3(f)}{2} \sum_{i \in \scrpt I} \Ee |X_i|^3 \, ,
\\
 \label{eq:Ind:CLT2}
 \bigl| \Ee \bigl[ f(W) \bigr] - \Normal f \bigr|
 &\le
 M_2(f) \sum_{i \in \scrpt I} \Ee \Bigl[ |X_i|^2 \min \bigl\{ 2.5, \> 0.94 \, |X_i| \bigr\} \Bigr] \, ,
%
% R:
% 
% 3*sqrt(2*pi)/8
% ## [1] 0.9399856
%
\\
 \label{eq:Ind:CLT1}
 \bigl| \Ee \bigl[ f(W) \bigr] - \Normal f \bigr|
 &\le
 M_1(f) \sum_{i \in \scrpt I} \Ee \Bigl[ |X_i|^2 \min \bigl\{
    4.5, \> \bigl( 11.1 + 0.83 \log d \, \bigr) \, |X_i| \bigr\} \Bigr] \, .
% 
% R:
%
% 5/2*1.8
% ## [1] 4.5
% 3/2*3.58 + 13/8*3.5
% ## 11.0575
% 3/2*0.55
% ## [1] 0.825
% 
\end{align}

\section{Proofs}
\label{sc:Prf}

\subsection{Assumptions~\ref{ass:size}, \ref{ass:prox} and \ref{ass:zero}}
\label{ssc:sizeproxzero}

Here, we prove Proposition~\ref{pr:size2zero} and derive stronger formulations
of Assumptions~\ref{ass:size} and \ref{ass:zeroalt}, which will be necessary
in the proofs of Theorems~\ref{th:CLT:size} and \ref{th:CLT:zero}.

\begin{proposition}
\label{pr:size:abs}
Assume \ref{ass:size} and take a measurable function $ f \Colon \RR^d \to \RR $ with
\newline
$ \int_{\Xi} \Ee_\xi \bigl| f(V_\xi) \bigr| \, |\mu|(\dl \xi) < \infty $.
Then we have $ \Ee \bigl| f(W) W \bigr| < \infty $ and \eqref{eq:size} remains true.
\end{proposition}

\begin{proof}
For each $ n \in \NN $, define $ f_n(w) := f(w) \One(|f(w)| \le n) $.
Next, for each $ n \in \NN $ and each $ u \in \RR^d $, define
$ f_{n,u}(w) := f_n(w) $ if $ \scalp{w}{u} \ge 0 $ and $ f_{n,u}(w) := - f_n(w) $ if $ \scalp{w}{u} < 0 $.
Observe that the functions $ f_n $ and $ f_{n,u} $ are measurable and bounded, so that \eqref{eq:size}
applies with $ f_n $ or $ f_{n,u} $ in place of $ f $. As a result, we have
\begin{align*}
 \Ee \bigl| f_n(W) \scalp{W}{u} \bigr|
 &=
 \Ee \bigl[ f_{n,u}(W) \scalp{W}{u} \bigr] \\
 &=
 \Scalp[bigg]{\int_{\Xi} \Ee_\xi \bigl[ f_{n,u}(V_\xi) \bigr] \, \mu(\dl \xi)}{u} \\
 &\le
 \int_{\Xi} \Ee_\xi \bigl| f(V_\xi) \bigr| \, |\mu|(\dl \xi) \\
 &<
 \infty
 \, .
\end{align*}
Noting that the functions $ f_n $ converge pointwise to $ f $, and applying Fatou's lemma, we
find that $ \Ee \bigl| f(W) \scalp{W}{u} \bigr| < \infty $ for all $ u \in \RR^d $.
Therefore, $ \Ee \bigl| f(W) W \bigr| < \infty $. Now we can apply the dominated convergence
theorem to the counterparts of \eqref{eq:size} with $ f_n $ in place of $ f $, which
implies that \eqref{eq:size} remains true.
\end{proof}

\begin{lemma}
\label{lm:proxE}
Under \ref{ass:prox}, we have $ \Ee |V_\xi| \le \Ee |W| + |\nu_\xi|(\Xi) $
for all $ \xi \in \Xi $.
\end{lemma}

\begin{proof}
Take $ \eps > 0 $, let $ f(w) := \sqrt{\eps^2 + |w|^2} $ and
compute $ \nabla f(w) = \frac{w}{\sqrt{\eps^2 + |w|^2}} $. Clearly, $ |\nabla f(w)| \le 1 $.
By Assumption~\ref{ass:prox}, we have
\begin{align*}
 \Ee_\xi |V_\xi|
 \le
 \Ee_\xi \bigl[ f(V_\xi) \bigr]
 &=
 \Ee \bigl[ f(W) \bigr]
   -
 \int_\Xi \Scalp[big]%
   {\Ee_\xi \bigl[ \nabla f(V_\xi) \bigr]}{\nu_\xi(\dl \eta)} \\
 &\le
 \eps + \Ee |W| + |\nu_\xi|(\Xi) \, .
\end{align*}
Letting $ \eps $ to zero, we obtain the desired inequality.
\end{proof}

\begin{lemma}
\label{lm:proxBd}
Assume \ref{ass:prox}, let $ D_\xi f := \Ee \bigl[ f(W) \bigr] - \Ee_\xi \bigl[ f(V_\xi) \bigr] $,
and recall \eqref{eq:betasize1-2} and \eqref{eq:betasize2-3}. Take $ \xi \in \Xi $ and a function
$ f \Colon \RR^d \to \RR $. If either $ f $ is measurable and bounded or $ M_1(f) < \infty $, then
\begin{equation}
\label{eq:proxBd0}
 |D_\xi f| \le 2 \, M_0(f) \, .
\end{equation}
Next, if $ f $ is continuously differentiable with $ M_0(f) < \infty $ or $ M_1(f) < \infty $, then
\begin{equation}
\label{eq:proxBd1}
 |D_\xi f| \le M_1(f) \, \beta_1^{(\xi)} \, .
\end{equation}
Finally, if $ f $ is twice continuously differentiable with $ M_1(f) < \infty $, then
\begin{equation}
\label{eq:proxBd2}
 |D_\xi f| \le \bigl| \Ee \bigl[ \nabla f(W) \bigr] \bigr| \, \beta_1^{(\xi)} + M_2(f) \, \beta_2^{(\xi)} \, .
\end{equation}
All the bounds apply under the convention $ 0 \cdot \infty = \infty \cdot 0 = 0 $.
\end{lemma}

\begin{remark}
Under the condition specified for each particular bound, all underlying expectations exist.
This follows from Remark~\ref{rk:MBd} and Lemma~\ref{lm:proxE}.
\end{remark}

\begin{proof}[Proof of Lemma~\ref{lm:proxBd}]
The bound~\eqref{eq:proxBd0} is immediate. To prove \eqref{eq:proxBd1}, assume first that
$ M_1(f) < \infty $. Combining \eqref{eq:prox} with Proposition~\ref{pr:IntVectMeasEst},
we find that
\begin{equation}
\label{eq:proxBdiint}
 |D_\xi f|
 \le
 \int_\Xi \bigl| \Ee_\eta \bigl[ \nabla f(V_\eta) \bigr] \bigr| \, |\nu_\xi|(\dl \eta)
\end{equation}
and \eqref{eq:proxBd1} follows. The latter is trivial if $ M_1(f) = \infty $ and
$ \beta_1^{(\xi)} > 0 $. If $ \beta_1^{(\xi)} = 0 $, then, by \eqref{eq:prox},
we have $ \Ee_\xi \bigl[ \tilde f(V_\xi) \bigr] = \Ee \bigl[ \tilde f(W) \bigr] $
for all continuously differentiable $ \tilde f $ with bounded derivative. As a
result, $ V_\xi $ and $ W $ have the same distribution and \eqref{eq:proxBd1} again
follows.

Applying \eqref{eq:proxBd1} with the function $ f_u(w) := \scalp{\nabla f}{u} $
in place of $ f $, where $ u \in \RR^d $, and with $ \eta $ in place of $ \xi $, we obtain
$ \bigl| \Scalp{\Ee_\eta \bigl[ \nabla f(V_\eta) \bigr]}{u} \bigr| \le
\bigl| \Scalp{\Ee \bigl[ \nabla f(W) \bigr]}{u} \bigr| + M_2(f) \, |u| \, \beta_1^{(\eta)} $.
Taking the supremum over $ |u| \le 1 $, we derive
$ \bigl| \Ee_\eta \bigl[ \nabla f(V_\eta) \bigr] \bigr| \le \bigl| \Ee \bigl[ \nabla f(W) \bigr] \bigr|
+ M_2(f) \, \beta_1^{(\eta)} $. Plugging into \eqref{eq:proxBdiint}, \eqref{eq:proxBd2} follows,
completing the proof.
\end{proof}

\begin{proof}[Proof of Proposition~\ref{pr:size2zero}]
First, we show that $ \Ee \bigl| f(W) W \bigr| < \infty $ and that \eqref{eq:size} still applies for
a continuously differentiable function $ f $ with bounded derivative. By Proposition~\ref{pr:size:abs},
it suffices to check that $ \int_{\Xi} \Ee_\xi \bigl| f(V_\xi) \bigr| \, |\mu|(\dl \xi)
< \infty $. Since $ |f(w)| \le |f(0)| + |w| \, M_1(f) $, it suffices to check that
$ \int_\Xi \Ee_\xi |V_\xi| \, |\mu|(\dl \xi) < \infty $.
However, this follows from the finiteness of $ \beta_2 $ by Lemma~\ref{lm:proxE}.

As $ \tilde f(w) := \sqrt{1 + |w|^2} $ is countinuously differentiable with bounded derivative,
$ \Ee \bigl| \tilde f(W) W \bigr| $ must be finite. Therefore, $ \Ee |W|^2 $ is finite, too.

Combining \eqref{eq:size} with the fact that $ \mu(\Xi) = 0 $ (which follows from Remark~\ref{rk:sizeE}
and the assumption $ \Ee W = 0 $), we obtain
\[
 \Ee \bigl[ f(W) W \bigr]
 =
 \int_{\Xi} \Bigl( \Ee_\xi \bigl[ f(V_\xi) \bigr] - \Ee \bigl[ f(W) \bigr] \Bigr)
 \mu(\dl \xi)
 \, .
\]
Applying \ref{ass:prox} and \eqref{eq:ActPureTens}, and recalling Definition~\ref{df:Phisuccmeas},
we rewrite this as
\begin{align*}
 \Ee \bigl[ f(W) W \bigr]
 &=
 - \int_{\Xi} \int_\Xi \Scalp[Big]%
   {\Ee_\eta \bigl[ \nabla f(V_\eta) \bigr]}%
   {\nu_\xi(\dl \eta)}
 \, \mu(\dl \xi) \\
 &=
 - \int_{\Xi} \int_\Xi
   \bigl( \mu(\dl \xi) \otimes \nu_\xi(\dl \eta) \bigr)
   \Ee_\eta \bigl[ \nabla f(V_\eta) \bigr]
 \, .
\end{align*}
A standard argument shows that $ \int_\Xi h \,\dl \breve \mu = \int_\Xi \int_\Xi h(\eta) \,
\nu_\xi(\dl \eta) \, \mu(\dl \xi) $ for all bounded measurable functions $ h $.
Property~\ref{ass:zero} now follows.
\end{proof}

\begin{proposition}
\label{pr:zeroaltStr}
Assume \ref{ass:zeroalt} and take a non-decreasing function $ h \Colon [0, \infty) \to [0, \infty) $,
such that
\begin{equation}
\label{eq:zeroaltStr:cond}
 \int_\Xi \Ee_\xi \bigl[ h(|V_\xi|) \bigr] \,
 |\breve \mu|_\wedge(\dl \xi) < \infty \, .
\end{equation}
Let $ F \Colon \RR^d \to \RR^d $ be a continuously differentiable vector function, such that
$ |\nabla F|_\vee(w) \le h(|w|) $ for all $ w \in \RR^d $.
Then $ \Ee \bigl| \scalp{F(W)}{W} \bigr| < \infty $ and \eqref{eq:zeroalt} remains true.
\end{proposition}

\begin{proof}
For each $ n \in \NN $, define function $ \psi_n \Colon [0, \infty) \to [0, 1] $
as $ \psi_n(t) := 1 $ for $ t \le n $, $ \psi_n(t) := 1 - \frac{1}{2 n^2} (t - n)^2 $
for $ n \le t \le 2n $, $ \psi_n(t) := \frac{1}{2 n^2} (t - 3n)^2 $ for $ 2n \le t \le 3n $
and $ \psi_n(t) := 0 $ for $ t \ge 3n $. Observe that $ \psi_n $ is well-defined and
that for each fixed $ n $, the expression $ t \, \psi_n(t) $ is bounded in $ t $.
Differentiating, we obtain $ \psi'_n(t) = 0 $ for $ t \le n $, $ \psi'_n(t) = \frac{1}{n^2} (t - n) $
for $ n \le t \le 2n $, $ \psi'_n(t) = \frac{1}{n^2} (t - 3n) $ for $ 2n \le t \le 3n $
and $ \psi'_n(t) = 0 $ for $ t \ge 3n $. Thus, $ \psi $ is continuously differentiable
and observe that the expression $ t \, |\psi'_n(t)| $ is uniformly bounded in $ t $ and $ n $.

Now let $ F_n(w) := F \bigl( \psi_n(|w|) \, w \bigr) $. Identifying $ 2 $-tensors with
linear transformations (see Section~\ref{ssc:Ten}, in particular \eqref{eq:ActTensScalp})
and applying the chain rule, compute
\[
 \nabla F_n(w) = \nabla F \bigl( \psi_n(|w|) \, w \bigr) \left(
  \frac{\psi'_n(|w|)}{|w|} \, w \otimes w + \psi_n(|w|) \, \Id_d \right)
\]
and notice that $ F_n $ is differentiable at the origin because the first term vanishes
for $ |w| \le n $.
Applying \eqref{eq:ActTensScalp}, \eqref{eq:ActPureTens} and again \eqref{eq:ActTensScalp} in
turn, we obtain
\begin{align*}
 \Scalp[big]{\nabla F \bigl( \psi_n(|w|) \, w \bigr)(w \otimes w)}{u \otimes v}
 &=
 \Scalp[big]{\nabla F \bigl( \psi_n(|w|) \, w \bigr)(w \otimes w) v}{u} \\
 &=
 \Scalp[big]{\nabla F \bigl( \psi_n(|w|) \, w \bigr) w}{u} \scalp{v}{w} \\
 &=
 \Scalp[big]{\nabla F \bigl( \psi_n(|w|) \, w \bigr)}{u \otimes w} \scalp{v}{w}
 \, .
\end{align*}
Therefore,
\[
 \scalp{\nabla h_n(w)}{u \otimes v}
 =
 \frac{\psi'_n(|w|)}{|w|} \, \Scalp[big]{\nabla F \bigl( \psi_n(|w|) \, w \bigr)}{u \otimes w} \scalp{v}{w}
  +
 \psi_n(|w|) \, \Scalp[big]{\nabla F \bigl( \psi_n(|w|) \, w \bigr)}{u \otimes v}
 \, .
\]
Taking the supremum over $ u $ and $ v $ with $ |u|, |v| \le 1 $, we obtain
\[
 |\nabla F_n(w)|_\vee
 \le
 \bigl( |w| \, |\psi'_n(|w|)| + \psi_n(|w|) \bigr) \, \bigl| \nabla F \bigl( \psi_n(|w|) \, w \bigr) \bigr|_\vee
 \, .
\]
Since $ t \, |\psi'(t)| $ is uniformly bounded in $ t $ and $ n $ and since $ h $ is non-decreasing,
there exists a constant $ C $, such that $ |\nabla F_n(w)|_\vee \le C \, h(|w|) $ for all $ n $ and $ w $.

For each fixed $ n $, the expression $ \psi_n(|w|) \, w $ is bounded in $ w \in \RR^d $. Since
$ F $ is continuously differentiable, $ |\nabla F_n(w)|_\vee $ is also bounded in $ w \in \RR^d $.
Therefore, \eqref{eq:zeroalt} holds true with $ F_n $ in place of $ F $.

Now observe that the functions $ h_n $ converge pointwise to $ F $ and that the functions
$ \nabla F_n $ converge pointwise to $ \nabla F $ as well. Recalling \eqref{eq:zeroaltStr:cond}
and applying the dominated convergence theorem, we obtain
\begin{equation}
\label{eq:zeroaltStr:lim}
 \lim_{n \to \infty} \Ee \bigl[ \scalp{F_n(W)}{W} \bigr]
 =
 \lim_{n \to \infty} \int_\Xi \Scalp[big]{\Ee_\xi
   \bigl[ \nabla F_n(V_\xi) \bigr]}{\breve \mu(\dl \xi)}
 =
 \int_\Xi \Scalp[big]{\Ee_\xi
   \bigl[ \nabla F(V_\xi) \bigr]}{\breve \mu(\dl \xi)}
 \, .
\end{equation}
Now take another continuously differentiable vector function $ \tilde F \Colon \RR^d \to \RR^d $,
such that $ |\nabla \tilde F|_\vee(w) \le h(|w|) $ and, in addition, $ \scalp{\tilde F(w)}{w} \ge 0 $
for all $ w \in \RR^d $. Letting $ \tilde F_n(w) := \tilde F \bigl( \psi_n(|w|) \, w \bigr) $,
observe that we also have $ \scalp{\tilde F_n(w)}{w} \ge 0 $ for all $ w \in \RR^d $.
Fatou's lemma along with \eqref{eq:zeroaltStr:lim} with $ \tilde h_n $ and $ \tilde F $
in place of $ F_n $ and $ F $ implies
\begin{equation}
\label{eq:zeroaltStr:Fin}
 \Ee \bigl[ \scalp{\tilde F(W)}{W} \bigr]
 \le
 \int_\Xi \Scalp[big]{\Ee_\xi
   \bigl[ \nabla \tilde F(V_\xi) \bigr]}{\breve \mu(\dl \xi)}
 \le
 \int_\Xi \Ee_\xi \bigl[ h(|V_\xi|) \bigr]
 |\breve \mu|_\wedge(\dl \xi)
 <
 \infty
 \, .
\end{equation}
Now put $ \tilde F(w) := \frac{w}{|w|} \int_0^{|w|} \tilde h(t) \,\dl t $, where
$ \tilde h(t) := \frac{1}{3} \bigl( h(t) - h_0 \bigr) $ and
$ h_0 := \lim_{s \downarrow 0} h(s) $. For $ w \ne 0 $, compute
\[
 \nabla \tilde F(w)
 =
 \left( \frac{\Id_d}{|w|} - \frac{w \otimes w}{|w|^3} \right) \int_0^{|w|} \tilde h(t) \,\dl t
  +
 \tilde h(|w|) \, \frac{w \otimes w}{|w|^2}
\]
and estimate
\[
 |\nabla \tilde F(w)|_\vee
 \le
 \frac{2}{|w|} \int_0^{|w|} \tilde h(t) \,\dl t + \tilde h(|w|)
 \le
 3 \, \tilde h(|w|)
 \le
 h(|w|) \, ;
\]
the second inequality is true because $ h $ is nondecreasing. From the above, it also follows
that $ \tilde F $ is continuously differentiable at the origin if we put $ \tilde F(0) := 0 $.
Now compute
\[
 \scalp{\tilde F(w)}{w}
 =
 |w| \int_0^{|w|} \tilde h(t) \,\dl t
 =
 \frac{|w|}{3} \int_0^{|w|} h(t) \,\dl t - \frac{h_0 |w|^2}{3}
\]
and estimate
\begin{align*}
 |F_n(w)|
 &=
 F \bigl( \psi_n(|w|) \, w \bigr)
 \le
 |F(0)| + \int_0^{\psi_n(|w|) \, |w|} h(t) \,\dl t
 \le
 |F(0)| + \int_0^{|w|} h(t) \,\dl t \, ,
\\
 |\scalp{F_n(w)}{w}|
 &\le
 |F_n(w)| \, |w|
 \le
 |F(0)| \, |w| + |w| \int_0^{|w|} h(t) \,\dl t
 \le
 |F(0)| \, |w| + h_0 \, |w|^2 + 3 \scalp{\tilde F(w)}{w} \, .
\end{align*}
Recalling \eqref{eq:zeroaltStr:Fin}, it follows that the sequence of random variables $ \scalp{F_n(W)}{W} $
is dominated by a non-negative random variable with finite expectation. Applying the
dominated convergence theorem and combining with \eqref{eq:zeroaltStr:lim},
the finiteness of $ \Ee \bigl| \scalp{F(W)}{W} \bigr| $ along with \eqref{eq:zeroalt}
follows.
\end{proof}

\subsection{Gaussian smoothing}

Gaussian smoothing will be one of the key tools to prove
Theorems~\ref{th:CLT:size} and \ref{th:CLT:zero}.
Let $ \phi_d $ be the density of the standard $ d $-variate normal density, i.~e.,
$ \phi_d(z) = (2 \pi)^{- d/2} \exp(- |z|^2/2) $.  For $ \eps \ge 0 $
and a map $ F \Colon \RR^d \to V $, where $ V $ is a finite-dimensional vector space,
define
\begin{equation}
\label{eq:GaussSm}
 \Normal_\eps F(w) := \int_{\RR^d} F(w + \eps z) \, \phi_d(z) \,\dl z \, .
\end{equation}
Notice that $ \Normal_0 F = F $ and $ \Normal_1 F(0) = \Normal h $.
Next, define constants $ c_0, c_1, c_2, \ldots $ as
\begin{equation}
\label{eq:cs}
 c_s := \int_{- \infty}^\infty |\phi^{(r)}_1(z)| \,\dl z \, .
\end{equation}
Observe that
\[
 \int_{- \infty}^\infty \bigl| \scalp{\phi^{(s)}_d(z)}{u^{\otimes s}} \bigr| \,\dl z
 \le
 c_s |u|^s
\]
and compute
\begin{equation}
\label{eq:csComp}
 c_0 = 1 \, , \quad
 c_1 = \frac{2}{\sqrt{2 \pi}} \, , \quad
 c_2 = \frac{4}{\sqrt{2 \pi e}} \, , \quad
 c_3 = \frac{2 + 8 \, e^{-3/2}}{\sqrt{2 \pi}}
 \, .
\end{equation}

\begin{lemma}
\label{lm:GaussSm}
Let $ \eps > 0 $. If $ f \Colon \RR^d \to \RR $ is either measurable and bounded or $ M_r(f) < \infty $
for some $ r \in \NN $, then $ \Normal_\eps |f|(w) < \infty $ for all $ w \in \RR^d $ and
$ \Normal_\eps f $ is infinitely differentiable. In addition, we have
\[
 M_{r+s}(\Normal_\eps f) \le \frac{c_s}{\eps^s} \, M_r(f)
\]
for all $ r \in \NN $ and all $ s \in \NN \cup \{ 0 \} $.
\end{lemma}

\begin{proof}
First, $ \Normal_\eps |f| $ is finite by Remark~\ref{rk:MBd} and the fact that
$ \int_{\RR^d} |z|^r \phi_d(z) \,\dl z $ is finite.
Substituting $ z = y - w/\eps $, we rewrite \eqref{eq:GaussSm} as
\begin{equation}
\label{eq:GaussSm:base}
 \Normal_\eps f(w) = \int_{\RR^d} f(\eps y) \, \phi_d \left( y - \frac{w}{\eps} \right) \dl y \, .
\end{equation}
Differentiating \eqref{eq:GaussSm:base} under the integral sign and substituting back, we obtain
\begin{align*}
 \nabla^s \Normal_\eps f(w)
 &=
 \frac{(-1)^s}{\eps^s} \int_{\RR^d} f(\eps y) \, \nabla^s \phi_d \left( y - \frac{w}{\eps} \right) \dl y \\
 &=
 \frac{(-1)^s}{\eps^s} \int_{\RR^d} f(w + \eps z) \, \nabla^s \phi_d(z) \, \dl z
 \, .
\end{align*}
Further differentiation under the differential sign gives
\[
 \nabla^{r+s-1} \Normal_\eps f(w)
 =
 \frac{(-1)^s}{\eps^s} \int_{\RR^d} \nabla^{r-1} f(w + \eps z) \otimes \nabla^s \phi_d(z) \, \dl z
 \, .
\]
The verification of the validity of the differentiation under the integral sign is left to the reader
as an exercise. Consequently,
\begin{align*}
 \Bigl| \Scalp[big]{\nabla^{r+s-1} \Normal_\eps f(x) - \nabla^{r+s-1} \Normal_\eps f(y)}{u^{\otimes (r+s-1)}} \Bigr|
 &\le
 \frac{|x - y|}{\eps^s} \, |u|^{r-1} \, M_r(f)
 \int_{\RR^d} \bigl| \scalp{\nabla^s \phi_d(z)}{u^{\otimes s}} \bigr| \, \dl z \\
 &\le
 c_s \, \frac{|x - y|}{\eps^s} \, |u|^{r+s-1} \, M_r(f)
 \, .
\end{align*}
By Proposition~\ref{pr:InjSym}, this implies
\[
 \Bigl| \nabla^{r+s-1} \Normal_\eps f(x) - \nabla^{r+s-1} \Normal_\eps f(y) \Bigr|_\vee
 \le
 c_s \, \frac{|x - y|}{\eps^s} \, M_r(f)
 \, .
\]
The result is now immediate.
\end{proof}

\subsection{Bounds on the Stein expectation}

In this subsection, we turn to Stein's method, which will
be implemented in view of the proof of Lemma~1 of Slepian~\citep{Slep}.
% , which is known as Slepian's lemma.
We recall
the procedure briefly; for an exposition, see R\"ollin~\citep{RolDim} and
Appendix~H of Chernozhukov, Chetverikov and Kato~\citep{CCK13sup}. Recalling
the definition of $ M_r(f) $ from \eqref{eq:Mr}, take a function
$ f \Colon \RR^d \to \RR $ with $ M_r(f) < \infty $ for some $ r \in \NN $.
For $ 0 \le \alpha \le \pi/2 $, define
\begin{equation}
\label{eq:Ualpha}
 \Slep_\alpha f(w)
 :=
 \Normal_{\sin \alpha} f(w \cos \alpha)
 =
 \int_{\RR^d} f(w \cos \alpha + z \sin \alpha) \, \phi_d(z) \,\dl z
\end{equation}
In particular, $ \Slep_0 f = f $ and $ \Slep_{\pi/2} f = \Normal f $.
By Lemma~\ref{lm:GaussSm}, $ \Slep_\alpha f $ is defined everywhere and is
infinitely differentiable.

For a random variable $ W $, $ \Ee \bigl[ \Slep_\alpha f(W) \bigr] $ can be regarded as an
interpolant between $ \Ee \bigl[ f(W) \bigr] $ and $ \Normal f $. A straightforward
calculation shows that
\[
 \frac{\dl}{\dl \alpha} \, \Slep_\alpha f(w) = \Stein \Slep_\alpha f(w) \tan \alpha
 \, ,
\]
where $ \Stein $ denotes the \emph{Stein operator}:
\begin{equation}
\label{eq:SteinOp}
 \Stein f(w) := \Delta f(w) - \scalp{\nabla f(w)}{w}
\end{equation}
and where $ \Delta $ denotes the Laplacian. Integrating over $ \alpha $ and
taking expectation, we find that
\begin{equation}
\label{eq:SteinSlep}
 \Ee \Slep_\eps f(W) - \Normal f
 =
 - \int_\eps^{\pi/2} \Ee \bigl[ \Stein \Slep_\alpha f(W) \bigr] \tan \alpha \,\dl \alpha
 \, ,
\end{equation}
for all $ 0 \le \eps \le \pi/2 $. More precisely, if
$ \int_\eps^{\pi/2} \Ee \bigl| \Stein \Slep_\alpha f(W) \bigr| \tan \alpha \,\dl \alpha $
is finite, then, by Fubini's theorem, $ \Ee \bigl| \Slep_\eps f(W) \bigr| $ is also finite and
\eqref{eq:SteinSlep} is true.

\begin{lemma}
\label{lm:DerUalpha}
Let $ r \in \NN $, $ s \in \NN \cup \{ 0 \} $ and $ 0 < \alpha \le \pi/2 $. Then any function
$ f \Colon \RR^d \to \RR $ with $ M_r(f) < \infty $ satisfies
\begin{align}
 \label{eq:DerUalpha}
 M_{r+s}(\Slep_\alpha f)
 \le
 c_s \, \frac{\cos^{r+s} \alpha}{\sin^s \alpha} \, M_r(f)
 \, , \\
 \label{eq:NDerUalpha}
 \bigl| \Normal \, \nabla^{r+s} \Slep_\alpha f \bigr|_\vee
 \le
 c_s \, M_r(f) \cos^{r+s} \alpha
 \, .
\end{align}
\end{lemma}

\begin{proof}
Letting $ f_\alpha(w) := f(w \cos \alpha) $, observe that
$ \Slep_\alpha f = \Normal_{\tan \alpha} f_\alpha $. By Lemma~\ref{lm:GaussSm},
we have $ M_{r+s}(\Slep_\alpha f) \le c_s \, M_r(f_\alpha) \cot^s \alpha
\le c_s \, M_r(f) \cos^r \alpha \cot^s \alpha $, proving \eqref{eq:DerUalpha}.
To derive \eqref{eq:NDerUalpha}, write
$ \Normal \, \nabla^{r+s} \Slep_\alpha f = \Normal_1 \, \nabla^{r+s} \Slep_\alpha f(0) =
\nabla^{r+s} \, \Normal_1 \Slep_\alpha f(0) $
and observe that $ \Normal_1 \Slep_\alpha f = \Normal_1 \Normal_{\tan \alpha} f_\alpha
= \Normal_{1/\cos \alpha} f_\alpha $. Again by Lemma~\ref{lm:GaussSm}, we have
$ M_{r+s}(\Normal_1 \Slep_\alpha f) \le c_s \, M_r(f_\alpha) \cos^s \alpha $\newline
$ \le c_s \, M_r(f) \cos^{r+s} \alpha $.
Combining with Remark~\ref{rk:Rademacher}, we obtain \eqref{eq:DerUalpha}.
\end{proof}

\noindent
Now we gradually turn to main result of this subsection, Lemma~\ref{lm:zero:EStein}. We first need
some results concerning finiteness of certain integrals.

\begin{lemma}
\label{lm:IntSteinFin}
Let $ r \in \NN $. Suppose that $ \Ee |W|^r < \infty $ and take a function
$ f \Colon \RR^d \to \RR $ with $ M_r(f) < \infty $. Then:
\begin{enumerate}
\refstepcounter{enumi}\refitem{$(\theenumi)$}\label{lm:IntSteinFin:f}
If $ r \ge 2 $, then $ \Ee \bigl| \Stein f(W) \bigr| < \infty $.
\refstepcounter{enumi}\refitem{$(\theenumi)$}\label{lm:IntSteinFin:Uaf}
$ \int_0^{\pi/2} \Ee \bigl| \Stein \Slep_\alpha f(W) \bigr| \tan \alpha \,\dl \alpha < \infty $
(this statement applies for either $ r \in \NN $).
\end{enumerate}
\end{lemma}

\begin{proof}
Clearly, $ \bigl| \Stein f(w) \bigr| \le d \, \bigl| \nabla^2 f(w) \bigr|_\vee +
|\nabla f(w)| \, |w| $ for all $ f $ with $ M_2(f) < \infty $. Let $ \tilde f_\alpha := \Slep_\alpha f $.
First, take $ r = 1 $. By Lemma~\ref{lm:DerUalpha},
we have $ M_1(\tilde f_\alpha) \le M_1(f) \cos \alpha $ and
$ M_2(\tilde f_\alpha) \le c_1 \, M_1(f) \, \frac{\cos^2 \alpha}{\sin \alpha} $.
As a result, we have
\[
 \Ee \bigl| \nabla^2 \tilde f_\alpha(W) \bigr|_\vee
 \le
 c_1 \, M_1(f) \, \frac{\cos^2 \alpha}{\sin \alpha}
 \kern 1.5em \text{and} \kern 1.5em
 \Ee \bigl[ |\nabla \tilde f_\alpha(W)| \, |W| \bigr]
 \le
 M_1(f) \Ee|W| \cos \alpha
 \, .
\]
Multiplying by $ \tan \alpha $ and integrating, we obtain the desired finiteness.

Now take $ r \ge 2 $. Observe that for each $ s = 0, 1, \ldots, r $, there exist constants
$ C_{s,r} $ and $ D_{s,r} $, such that $ |\nabla^s f(x)|_\vee \le C_{s,r} + D_{s,r} |x|^{r-s} $
for all $ x \in \RR^d $. Thus, $ \Ee \bigl| \Stein f(W) \bigr| \le
d \, C_{2,r} + \, D_{2,r} \Ee |W|^{r-2} + C_{1,r} \Ee |W| + D_{1,r} \Ee |W|^r < \infty $.

Next, similarly as in the proof of Lemma~\ref{lm:DerUalpha}, differentiation
under the integration sign gives $ \nabla^s \tilde f_\alpha(w) = \Ee \bigl[ \nabla^s f(w \cos \alpha +
Z \sin \alpha) \bigr] \cos^s \alpha $, where $ Z $ is a standard $ d $-variate normal random
vector. Therefore,
\[
 \bigl| \nabla^s \tilde f_\alpha(w) \bigr|_\vee
 \le
 C_{s,r} \cos^s \alpha
   +
 D_{s,r} \sum_{k=0}^{r-s} \binom{n-s}{k} |w|^k \Ee |Z|^{r-s-k} \cos^{s+k} \alpha \sin^{r-s-k} \alpha
 \, .
\]
Replacing $ w $ with $ W $ and taking expectation, we obtain in particular
\begin{align*}
 \Ee \bigl| \nabla^2 \tilde f_\alpha(W) \bigr|_\vee
 &\le
 C_{2,r} \cos^2 \alpha
   +
 D_{2,r} \sum_{k=0}^{r-2} \binom{n-2}{k} \Ee |W|^k \Ee |Z|^{r-2-k} \cos^{k+2} \alpha \sin^{r-k-2} \alpha
 \, ,
\\
 \Ee \bigl[ |\nabla \tilde f_\alpha(W)| \, |W| \bigr]
 &\le
 C_{1,r} \cos \alpha
   +
 D_{1,r} \sum_{k=0}^{r-1} \binom{n-1}{k} \Ee |W|^{k+1} \Ee |Z|^{r-k-1} \cos^{k+1} \alpha \sin^{r-k-1} \alpha
 \, .
\end{align*}
Multiplying by $ \tan \alpha $ and integrating, we again obtain the desired finiteness.
\end{proof}

\begin{lemma}
\label{lm:zero:EW3fin}
Assume \ref{ass:zeroalt} and \ref{ass:prox}. If the quantity $ \breve \beta_3 $ defined in \eqref{eq:betazero-3}
is finite, then
\newline
$ \int_\Xi \Ee_\xi |V_\xi| \,|\breve \mu|_\wedge(\dl \xi) $
and $ \Ee |W|^3 $ are finite, too.
\end{lemma}

\begin{proof}
The finiteness of
$ \int_\Xi \Ee_\xi |V_\xi| \,|\breve \mu|_\wedge(\dl \xi) $
follows from the finiteness of $ \breve \beta_3 $ by Lemma~\ref{lm:proxE}.
To derive the finiteness of $ \Ee |W|^3 $, apply Proposition~\ref{pr:zeroaltStr} with $ F(w) = |w| w $: clearly,
\newline
$ |\scalp{F(w)}{w}| = |w|^3 $. Noting that $ \nabla F(w) = \frac{w \otimes w}{|w|} + |w| \Id_d $,
we can set $ h(t) := 2 t $, so that the finiteness of
$ \int_\Xi \Ee_\xi |V_\xi| \,|\breve \mu|_\wedge(\dl \xi) $
implies \eqref{eq:zeroaltStr:cond} and the result follows.
\end{proof}

\begin{lemma}
\label{lm:zero:EStein}
Assume \ref{ass:St}, \ref{ass:zeroalt} and \ref{ass:prox}, and recall \eqref{eq:betasize1-2},
\eqref{eq:betasize2-3}, \eqref{eq:betazero-3} and \eqref{eq:betazero-23}. Take a three times
differentiable function $ f \Colon \RR^d \to \RR $. If either $ M_2(f) < \infty $,
or $ \breve \beta_3 < \infty $ and $ M_3(f) < \infty $, then
\begin{equation}
\label{eq:zero:EStein:beta23}
 \bigl| \Ee \bigl[ \Stein f(W) \bigr] \bigr| \le \breve \beta_{23} \bigl( 2 \, M_2(f), \, M_3(f) \bigr) \, .
\end{equation}
Moreover, if $ f $ is four times continuously differentiable with
$ M_2(f) < \infty $ and $ M_3(f) < \infty $, then
\begin{equation}
\label{eq:zero:EStein:beta234}
 \bigl| \Ee \bigl[ \Stein f(W) \bigr] \bigr|
 \le
 \int_\Xi \min \Bigl\{
   2 \, M_2(f) \, , \>
   \beta_1^{(\xi)} \, \bigl| \Ee \bigl[ \nabla^3 f(W) \bigr] \bigr|_\vee
     +
   \beta_2^{(\xi)} \, M_4(f)
 \Bigr\} \, |\breve \mu|_\wedge(\dl \xi)
 \, .
\end{equation}
Both bounds apply under the convention $ 0 \cdot \infty = \infty \cdot 0 = 0 $.
\end{lemma}

\begin{remark}
By Part~\ref{lm:IntSteinFin:f} of Lemma~\ref{lm:IntSteinFin} along with Lemma~\ref{lm:zero:EW3fin},
$ \Ee \bigl| \Stein f(W) \bigr| $ is finite in either case.
\end{remark}

\begin{proof}[Proof of Lemma~\ref{lm:zero:EStein}]
First, observe that \eqref{eq:zeroalt} applies with $ F = \nabla f $ under the either of the
specified conditions. If $ M_2(f) < \infty $, this is true by Assumption~\ref{ass:zeroalt}
itself. If $ \breve \beta_3 < \infty $ and $ M_3(f) < \infty $, we can apply Proposition~\ref{pr:zeroaltStr}.
The latter applies if we can estimate $ |\nabla^2 f(w)|_\vee \le h(|w|) $, where $ h $ is a non-decreasing function with
$ \int_\Xi \Ee_\xi \bigl[ h(|V_\xi|) \bigr]
|\breve \mu|_\wedge(\dl \xi) < \infty $. As $ |\nabla^2 f(w)|_\vee \le |\nabla^2 f(0)|_\vee + M_3(f) \, |w| $,
it suffices to see that
$ \int_\Xi \Ee_\xi |V_\xi| \,|\breve \mu|_\wedge(\dl \xi) < \infty $.
However, this is true by Lemma~\ref{lm:zero:EW3fin}.

Recalling the identification $ u \otimes v \equiv u v^\Trans $
from Subsection~\ref{ssc:Ten} and Remark~\ref{rk:zerovar}, and applying \ref{ass:St}, consider
\begin{equation}
\label{eq:zero:EStein:Delta}
 \Delta f(w)
 =
 \Scalp[big]{\nabla^2 f(w)}{\Id_d}
 =
 \Scalp[big]{\nabla^2 f(w)}{\breve \mu(\Xi)}
 =
 \int_\Xi \scalp{\nabla^2 f(w)}{\breve \mu(\dl \xi)}
 \, .
\end{equation}
Taking expectation and applying \eqref{eq:zeroalt}, we obtain
\begin{equation}
\label{eq:zero:EStein:D}
 \Ee \bigl[ \Stein f(W) \bigr] = \int_\Xi \scalp{D_\xi \nabla^2 f}{\breve \mu(\dl \xi)} \, ,
\end{equation}
where $ D_\xi $ is as in Lemma~\ref{lm:proxE}, i.~e.,
$ D_\xi g := \Ee \bigl[ g(W) \bigr] - \Ee_\xi \bigl[ g(V_\xi) \bigr] $.

Let $ u, v \in \RR^d $ and suppose that $ M_2(f) < \infty $.
The bounds \eqref{eq:proxBd0} and \eqref{eq:proxBd1} applied to the function
$ f_{u,v}(w) := \scalp{\nabla^2 f}{u \otimes v} $ yield $
 \bigl| \scalp{D_\xi \nabla^2 f}{u \otimes v} \bigr|
 \le
 \min \bigl\{ 2 \, M_2(f), M_3(f) \, \beta_1^{(\xi)} \bigr\}
$. Taking the supremum over $ |u|, |v| \le 1 $, we find that
$ |D_\xi \nabla^2 f|_\vee \le \min \bigl\{ 2 \, M_2(f), M_3(f) \, \beta_1^{(\xi)} \bigr\} $.
Plugging into \eqref{eq:zero:EStein:D} and applying Proposition~\ref{pr:IntVectMeasEst},
we obtain \eqref{eq:zero:EStein:beta23}. By a similar argument, \eqref{eq:zero:EStein:beta234}
can be derived from \eqref{eq:proxBd0} combined with \eqref{eq:proxBd2}.
\end{proof}

\subsection{Proofs of Theorems~\ref{th:CLT:size} and \ref{th:CLT:zero}}
\label{ssc:Prf}

\begin{proof}[Proof of Theorem~\ref{th:CLT:zero}]
First, we verify that under the given assumptions, first, $ \Ee \bigl| f(W) \bigr| $ and
$ \Normal |f| $ are finite, and, second, either \eqref{eq:SteinSlep} is valid
or the result is trivially true. The finiteness of $ \Normal |f| $ follows from
Remark~\ref{rk:MBd} and the fact that $ \int_{\RR^d} |z|^r \phi_d(z) \,\dl z $ is
finite. Now if $ M_1(f) $ or $ M_2(f) $ is finite, then, by Lemma~\ref{lm:IntSteinFin},
$ \int_0^{\pi/2} \Ee \bigl| \Stein \Slep_\alpha f(W) \bigr| \tan \alpha \,\dl \alpha $
is also finite. Consequently, $ \Ee \bigl| f(W) \bigr| $ is finite and
\eqref{eq:SteinSlep} is valid (see the comment below \eqref{eq:SteinSlep}).

It remains to show that if $ M_3(f) < \infty $, then, first, $ \Ee \bigl| f(W) \bigr| $
is finite, and, second, either \eqref{eq:SteinSlep} is valid or \eqref{eq:CLT3Xi} is
trivially true. If $ M_3(f) = 0 $, then $ f $ is a quadratic
function, so that $ \Ee \bigl| f(W) \bigr| $ must be finite and
$ \Ee \bigl[ f(W) \bigr] = \Normal f $; as a result, \eqref{eq:CLT3Xi} is
trivially true. The latter is also trivially true if the right hand side is
infinite. However, if $ M_3(f) > 0 $ and the right hand side is finite,
then $ \breve \beta_3 $ is also finite. By Lemma~\ref{lm:zero:EW3fin},
$ \Ee |W|^3 $ must then be finite; by Lemma~\ref{lm:IntSteinFin},
$ \int_0^{\pi/2} \Ee \bigl| \Stein \Slep_\alpha f(W) \bigr| \tan \alpha \,\dl \alpha $
must be finite. As a result, $ \Ee \bigl| f(W) \bigr| $ is finite and
\eqref{eq:SteinSlep} is valid.

First, observe that $ \Ee \bigl| f(W) \bigr| $ and $ \Normal |f| $ are finite and
$ \Ee \bigl[ f(W) \bigr] = \Normal f $ if either of $ M_1(f) $, $ M_2(f) $ or $ M_3(f) $
vanishes. Thus, \eqref{eq:CLT3Xi}--\eqref{eq:CLT1Xi} are all trivially true in this case.
From now, assume that $ M_r(f) > 0 $ for all $ r \in \{ 1, 2, 3 \} $. In addition,
to prove either of the inequalities~\eqref{eq:CLT3Xi}--\eqref{eq:CLT1Xi}, we can assume
that the right hand side is finite.

Thus, in either case where we have not proved the result yet, we can estimate
\begin{equation}
\label{eq:SteinSlepEst}
 \bigl| \Ee \bigl[ \tilde f_\eps(W) \bigr] - \Normal f \bigr|
 \le
 \int_\eps^{\pi/2} \bigl| \Ee \bigl[ \Stein \tilde f_\alpha(W) \bigr] \bigr|
 \tan \alpha \,\dl \alpha
\end{equation}
for all $ 0 \le \eps \le \pi/2 $; here, $ \tilde f_\alpha := \Slep_\alpha f $.

In order to derive \eqref{eq:CLT3Xi}, an application of
the second part of \eqref{eq:zero:EStein:beta23} along with \eqref{eq:DerUalpha} gives
\[
 \bigl| \Ee \bigl[ f(W) \bigr] - \Normal f \bigr|
 \le
 \breve \beta_3 \int_0^{\pi/2} M_3(\tilde f_\alpha) \tan \alpha \,\dl \alpha
 \le
 \breve \beta_3 \, M_3(f) \int_0^{\pi/2} \cos^2 \alpha \sin \alpha \,\dl \alpha
 =
 \frac{\breve \beta_3}{3} \, M_3(f)
 \, .
\]
To derive \eqref{eq:CLT2Xi}, observe that
\begin{align*}
 \bigl| \Ee \bigl[ f(W) \bigr] - \Normal f \bigr|
 &\le
 \int_0^{\pi/2} \breve \beta_{23} \bigl( 2 \, M_2(\tilde f_\alpha), \, M_3(\tilde f_\alpha) \bigr) \tan \alpha \,\dl \alpha \\
 & \kern -1em \null =
 \int_\Xi
   \int_0^{\pi/2}
     \min \Bigl\{ 2 \, M_2(\tilde f_\alpha), \, M_3(\tilde f_\alpha) \, \beta_1^{(\xi)} \Bigr\} \,
   \tan \alpha \,\dl \alpha
 \, |\breve \mu|_\wedge(\dl \xi) \\
 & \kern -1em \null \le
 \int_\Xi \min \biggl\{
   2 \int_0^{\pi/2} M_2(\tilde f_\alpha) \tan \alpha \,\dl \alpha
 \, , \>
   \beta_1^{(\xi)} \int_0^{\pi/2} M_3(\tilde f_\alpha) \tan \alpha \,\dl \alpha
 \biggr\} |\breve \mu|_\wedge(\dl \xi)
 \, .
\end{align*}
Applying \eqref{eq:DerUalpha}, recalling \eqref{eq:betazero-23} and integrating, we obtain
\begin{align*}
 \bigl| \Ee \bigl[ f(W) \bigr] - \Normal f \bigr|
 &\le
 M_2(f)
 \int_\Xi \min \biggl\{
   2 \int_0^{\pi/2} \cos \alpha \sin \alpha \,\dl \alpha
 \, , \>
   c_1 \, \beta_1^{(\xi)} \int_0^{\pi/2} \cos^2 \alpha \,\dl \alpha
 \biggr\} \, |\breve \mu|_\wedge(\dl \xi) \\
 &=
 M_2(f)
 \int_\Xi \min \biggl\{
   1
 \, , \>
   \frac{\pi}{4} \, c_1 \, \beta_1^{(\xi)}
 \biggr\} \, |\breve \mu|_\wedge(\dl \xi) \\
 &=
 \breve \beta_{23} \left( 1, \, \frac{\sqrt{2 \pi}}{4} \right)
 M_2(f)
 \, .
\end{align*}
Finally, we turn to \eqref{eq:CLT1Xi}. The latter is trivially true if $ M_1(f) = 0 $,
so that we can assume that $ M_1(f) > 0 $. Define
\begin{equation}
\label{eq:delta}
 \delta := \sup \left\{
   \frac{\bigl| \Ee \bigl[ f(W) \bigr] - \Normal f \bigr|}{M_1(f)} \> \sth \> 0 < M_1(f) < \infty
 \right\} \, .
\end{equation}
We first prove that $ \delta $ is finite. To justify this, observe that
$ \bigl| \Ee \bigl[ f(W) \bigr] - f(0) \bigr| \le M_1(f) \Ee|W| $ and
$ \bigl| \Normal f - f(0) \bigr| \le M_1(f) \Ee|Z| $. Since $ \Ee |W| $ and
$ \Ee |Z| $ are both finite, $ \delta $ must be finite, too.

Now take $ f \Colon \RR^d \to \RR $ with $ 0 < M_1(f) < \infty $. We begin with
smoothing: let $ 0 < \eps < \pi/2 $ and take a standard $ d $-variate normal
random vector $ Z $, independent of $ W $. Observe that
\begin{align*}
 \bigl| \Ee \bigl[ \tilde f_\eps(W) \bigr] - \Ee \bigl[ f(W) \bigr] \bigr|
 &\le
 M_1(f) \Ee \bigl| (1 - \cos \eps) W + Z \sin \eps \bigr| \\
 &\le
 M_1(f) \sqrt{\Ee \bigl| (1 - \cos \eps) W + Z \sin \eps \bigr|^2} \\
 &\le
 M_1(f) \sqrt{(1 - \cos \eps)^2 \Ee |W|^2 + \Ee |Z|^2 \sin^2 \eps} \\
 &=
 2 \sqrt{d} \, M_1(f) \sin \frac{\eps}{2}
 \, .
\end{align*}
Consequently,
\begin{equation}
\label{eq:CLT:zero:EstSm}
 \bigl| \Ee \bigl[ f(W) \bigr] - \Normal f \bigr|
 \le
 \bigl| \Ee \bigl[ \tilde f_\eps(W) \bigr] - \Normal f \bigr|
   +
 2 \sqrt{d} \, M_1(f) \sin \frac{\eps}{2}
 \, .
\end{equation}
Combining \eqref{eq:SteinSlepEst} with \eqref{eq:zero:EStein:beta234}, we obtain
\begin{equation}
\label{eq:CLT:zero:beta***}
 \bigl| \Ee \bigl[ \tilde f_\eps(W) \bigr] - \Normal f \bigr|
 \le
 \int_\Xi R_\xi \, |\breve \mu|_\wedge(\dl \xi)
 \, ,
\end{equation}
where
\begin{equation}
\label{eq:CLT:zero:Dxi}
 R_\xi
 :=
 \int_\eps^{\pi/2} \min \Bigl\{
   2 \, M_2(\tilde f_\alpha) \, , \>
   \beta_1^{(\xi)} \, \bigl| \Ee \bigl[ \nabla^3 \tilde f_\alpha(W) \bigr] \bigr|_\vee
     +
   M_4(\tilde f_\alpha) \, \beta_2^{(\xi)}
 \Bigr\} \tan \alpha \,\dl \alpha
 \, .
\end{equation}
Now take $ u \in \RR^d $ and let $ \tilde f_{\alpha; u} := \scalp{\nabla^3 \tilde f_\alpha}{u^{\otimes 3}} $.
Applying \eqref{eq:DerUalpha}, we can estimate
\begin{equation*}
\label{eq:CLT:zero:nabla3W:dir}
 \Bigl| \Scalp[big]{\Ee \bigl[ \nabla^3 \tilde f_\alpha(W) \bigr]}{u^{\otimes 3}} \Bigr|
 =
 \bigl| \Ee \bigl[ \tilde f_{\alpha; u}(W) \bigr] \Bigr|
 \le
 M_3(\tilde f_{\alpha; u})
 \le
 M_3(\tilde f_\alpha) \, |u|^3
 \, .
\end{equation*}
However, by \eqref{eq:delta}, we can also estimate
\[
\label{eq:CLT:zero:err}
 \Bigl| \Scalp[big]{\Ee \bigl[ \nabla^3 \tilde f_\alpha(W) \bigr] - \Normal \, \nabla^3 \tilde f_\alpha}{u^{\otimes 3}} \Bigr|
 \le
 \delta \, M_1(\tilde f_{\alpha; u})
 \le
 \delta \, M_4(\tilde f_\alpha) \, |u|^3
\]
and, consequently,
\begin{equation}
\label{eq:CLT:zero:nabla3W:Diff}
 \Bigl| \Scalp[big]{\Ee \bigl[ \nabla^3 \tilde f_\alpha(W) \bigr]}{u^{\otimes 3}} \Bigr|
 \le
 \Bigl(
   \Bigl| \Scalp[big]{\Normal \, \nabla^3 \tilde f_\alpha}{u^{\otimes 3}} \Bigr|_\vee
     +
   \delta \, M_4(\tilde f_\alpha)
 \Bigr) |u|^3
 \, .
\end{equation}
Combining \eqref{eq:CLT:zero:nabla3W:dir} and \eqref{eq:CLT:zero:nabla3W:Diff},
taking the supremum over all $ u $ with $ |u| \le 1 $, and applying \eqref{eq:DerUalpha}
and \eqref{eq:NDerUalpha}, we obtain
\begin{align*}
 \Bigl| \Ee \bigl[ \nabla^3 \tilde f_\alpha(W) \bigr] \Bigr|_\vee
 &\le
 \min \Bigl\{ 
    M_3(\tilde f_\alpha), \,
    \bigl| \Normal \, \nabla^3 \tilde f_\alpha \bigr|_\vee + \delta \, M_4(\tilde f_\alpha)
 \Bigr\} \\
 &\le
 M_1(f) \left( 
   c_2 \cos^3 \alpha
     +
   \min \left\{
     c_2 \, \frac{\cos^3 \alpha}{\sin^2 \alpha}, \,
     c_3 \, \delta \, \frac{\cos^4 \alpha}{\sin^3 \alpha}
   \right\}
 \right) \, .
\end{align*}
Plugging into \eqref{eq:CLT:zero:Dxi}, and estimating $ M_2(\tilde f_\alpha) $ and $ M_4(\tilde f_\alpha) $ by
means of \eqref{eq:DerUalpha}, we find that
\begin{equation*}
\begin{split}
 R_\xi
 &\le M_1(f)
 \int_\eps^{\pi/2} \min \biggl\{
   2 \, c_1 \cos \alpha \, , \>
   \beta_1^{(\xi)} \left(
     c_2 \cos^2 \alpha \sin \alpha
       +
     \min \left\{
       c_2 \, \frac{\cos^2 \alpha}{\sin \alpha}, \,
       c_3 \, \delta \, \frac{\cos^3 \alpha}{\sin^2 \alpha}
     \right\}
   \right) \\
 & \kern 5em \null
     +
   c_3 \, \beta_2^{(\xi)} \, \frac{\cos^3 \alpha}{\sin^2 \alpha}
 \biggr\} \, \dl \alpha
\\
 &\le
 M_1(f) \min \Biggl\{ 2 \, c_1 \, , \>
   \frac{c_2}{3} \, \beta_1^{(\xi)}
     +
   \beta_1^{(\xi)} \int_\eps^{\pi/2} \min \biggl\{
     c_2 \, \frac{\cos^2 \alpha}{\sin \alpha}, \,
     c_3 \, \delta \, \frac{\cos^3 \alpha}{\sin^2 \alpha}
   \biggr\} \,\dl \alpha \\
 & \kern 5em \null
     +
   \int_0^{\pi/2} \min \biggl\{
     2 \, c_1, \, 
     \frac{c_3 \, \beta_2^{(\xi)}}{\sin^2 \alpha}
   \biggr\} \cos \alpha \,\dl \alpha
 \Biggr\}
 \, .
\end{split}
\end{equation*}
The second integral can be estimated as
\begin{equation}
\label{eq:circum}
 \int_0^{\pi/2} \min \biggl\{
   2 \, c_1, \, 
   \frac{c_3 \, \beta_2^{(\xi)}}{\sin^2 \alpha}
 \biggr\} \cos \alpha \,\dl \alpha
 \le
 \int_0^\infty \min \biggl\{
   2 \, c_1, \, 
   \frac{c_3 \, \beta_2^{(\xi)}}{s^2}
 \biggr\} \,\dl s
 =
 2 \sqrt{2 \, c_1 c_3 \, \beta_2^{(\xi)}} \, ,
\end{equation}
and the first one can be estimated as
\begin{equation*}
\begin{split}
 \int_\eps^{\pi/2} \min \biggl\{
   c_2 \, \frac{\cos^2 \alpha}{\sin \alpha}, \,
   c_3 \, \delta \, \frac{\cos^3 \alpha}{\sin^2 \alpha}
 \biggr\} \,\dl \alpha
 &\le
 \int_\eps^{\pi/2} \min \biggl\{
   \frac{c_2}{2 \sin \frac{\alpha}{2}}, \,
   \frac{c_3 \, \delta}{4 \sin^2 \frac{\alpha}{2}}
 \biggr\} \cos \frac{\alpha}{2} \,\dl \alpha
\\
 &\le
 \int_{\sin (\eps/2)}^\infty \min \biggl\{
   \frac{c_2}{s}, \,
   \frac{c_3 \, \delta}{2 s^2}
 \biggr\} \,\dl s
\\
 &\le
 \left( \int_{\sin (\eps/2)}^{(c_3 \delta)/(2 c_2)} \frac{c_2}{s} \,\dl s \right)_+
   +
 \int_{(c_3 \delta)/(2 c_2)}^\infty \frac{c_3 \, \delta}{2 s^2} \,\dl s
\\
 &=
 c_2 \left[ 1 + \left( \log \frac{c_3 \delta}{2 c_2 \sin \frac{\eps}{2}} \right)_+ \right]
 \, .
\end{split}
\end{equation*}
Collecting everything together, we obtain
\begin{equation*}
 R_\xi
 \le
 M_1(f) \min \Biggl\{
   2 \, c_1, \,
   c_2  \, \beta_1^{(\xi)} \left[
     \frac{4}{3} + \left( \log \frac{c_3 \delta}{2 c_2 \sin \frac{\eps}{2}} \right)_+
   \right]
     +
   2 \sqrt{2 \, c_1 c_3 \, \beta_2^{(\xi)}}
 \Biggr\}
 \, .
\end{equation*}
Combining with \eqref{eq:CLT:zero:EstSm} and \eqref{eq:CLT:zero:beta***}, and recalling \eqref{eq:betazero-234},
we estimate the error in the normal approximation as
\begin{equation*}
 \bigl| \Ee \bigl[ f(W) \bigr] - \Normal f \bigr|
 \le
 M_1(f) \Biggl[
   2 \sqrt{d} \sin \frac{\eps}{2} + \breve \beta_{234} \left(
     2 \, c_1, \>
     c_2 \left[ \frac{4}{3} + \left( \log \frac{c_3 \delta}{2 c_2 \sin \frac{\eps}{2}} \right)_+ \right] , \,
     2 \sqrt{2 \, c_1 c_3}
   \right)
 \Biggr]
 \, .
\end{equation*}
Taking the supremum over $ f $, dividing by $ M_1(f) $ and choosing $ \eps := 2 \arcsin \frac{\delta}{18 \sqrt{d}} $
leads to the estimate
\[
 \delta \le \frac{\delta}{9} + \breve \beta_{234} \left(
   2 \, c_1, \,
   c_2 \left[ \frac{4}{3} + \log \frac{9 c_3 \sqrt{d}}{2 c_2} \right] , \,
   2 \sqrt{2 \, c_1 c_3}
 \right)
\]
Resolving the latter and recalling that $ \delta $ is finite, the result follows after
straightforward numerical computations.
%
% R:
% c1 <- 2/sqrt(2*pi)
% c2 <- 4/sqrt(2*pi*exp(1))
% c3 <- (2 + 8*exp(-3/2))/sqrt(2*pi)
% 9*c1/4
% ## [1] 1.79524
% 9/8*c2*(4/3 + log(9*c3/2/c2))
% ##[1] 3.573854
% 9*c2/16
% ## [1] 0.5444341
% 9/4*sqrt(2*c1*c3)
% ## [1] 3.492673
% 
\end{proof}

\begin{proof}[Proof of Theorem~\ref{th:CLT:size}]
We derive the result from Theorem~\ref{th:CLT:zero}.
Let $ \breve \mu $ be as in Proposition~\ref{pr:size2zero}.
By the latter, it satisfies Assumption~\ref{ass:zero}. Noting that $ \breve \beta_3 \le \beta_3 $,
$ \breve \beta_{23}(a, b) \le \beta_{23}(a, b) $ and $ \breve \beta_{234}(a, b, c) \le \beta_{234}(a, b, c) $,
the inequalities~\eqref{eq:CLT3xi}--\eqref{eq:CLT1xi} follow from
the inequalities~\eqref{eq:CLT3Xi}--\eqref{eq:CLT1Xi}.
\end{proof}

\section{Appendix: theoretical preliminaries, notation and conventions}
\label{sc:Prel}

Throughout this appendix, $ U $, $ U' $, $ V $, $ V' $, $ W $, $ W' $, $ Z $ and $ Z' $ will denote
vector spaces. Unless specified otherwise, all vector spaces will be assumed to be real and
finite-dimensional.

\subsection{Dual pairs of vector spaces}

\begin{definition}
A \emph{dual pair} of vector spaces is a triplet $ (V, V', \scalp{\cdot}{\cdot}) $,
where $ \scalp{\cdot}{\cdot} $ is a nonsingular \emph{pairing} between $ V $ and
$ V' $, i.~e., a bilinear functional $ V \times V' \to \RR $, such
that for each $ v \in V \setminus \{ 0 \} $, there exists $ v' \in V' $ with $ \scalp{v}{v'} \ne 0 $, and
that for each $ v' \in V' \setminus \{ 0 \} $, there exists $ v \in V $ with $ \scalp{v}{v'} \ne 0 $.
\end{definition}

Observe that if $ V' $ is the space of all linear functionals on $ V $ and $ \scalp{v}{v'} = v'(v) $,
then $ (V, V', \scalp{\cdot}{\cdot}) $ is a dual pair of vector spaces. This is true because
all linear functionals on subspaces can be extended to the whole space (a well known extension
to the infinite-dimensional case is the Hahn--Banach theorem, see Theorem~3.3 of Rudin~\citep{Rud3}).
Conversely, if $ (V, V', \scalp{\cdot}{\cdot}) $ is a dual pair of vector spaces, $ V' $ is
naturally isomorphic to the space of all linear functionals on $ V $ and vice versa. Thus,
$ V $ and $ V' $ are of the same dimension.

\begin{definition}
Let $ (V, V', \scalp{\cdot}{\cdot}) $ be a dual pair of vector spaces. The bases
$ e_1, \ldots, e_n $ of $ V $ and $ e'_1, \ldots, e'_n $ of $ V' $ are
dual with respect to the pairing $ \scalp{\cdot}{\cdot} $ if $ \scalp{e_i}{e'_j}
= \One(i = j) $ for all $ i $ and $ j $.
\end{definition}

\noindent
Observe that for each basis of $ V $, there exists a unique dual basis if $ V' $.

\begin{definition}
A dual pair of \emph{normed spaces} is a quintuplet $ (V, V', \scalp{\cdot}{\cdot}, |\cdot|, |\cdot|') $,
where $ (V, V', \scalp{\cdot}{\cdot}) $ is a dual pair of vector spaces, $ |\cdot| $ is a norm
on $ V $, $ |\cdot|' $ a norm on $ V' $, and $ |\cdot| $ and $ |\cdot|' $ are \emph{dual norms}
(with respect to the pairing $ \scalp{\cdot}{\cdot} $), i.~e.,
$ |v| = \sup \{ |\scalp{v}{v'}| \sth |v'|' \le 1 \} $ for all $ v \in V $ and
$ |v'|' = \sup \{ |\scalp{v}{v'}| \sth |v| \le 1 \} $ for all $ v' \in V' $.
\end{definition}

Observe that one of the assumptions on norms is sufficient: if
$ |v| = \sup \{ |\scalp{v}{v'}| \sth |v'|' \le 1 \} $ for all $ v \in V $, then also
$ |v'|' = \sup \{ |\scalp{v}{v'}| \sth |v| \le 1 \} $ for all $ v' \in V' $. This is
due to the Hahn--Banach theorem.

If $ V $ is an Euclidean space with a scalar product $ \scalp{\cdot}{\cdot} $ and the
underlying norm $ |\cdot| $, then $ (V, V, \scalp{\cdot}{\cdot}, |\cdot|, |\cdot|) $ is
a dual pair of normed spaces.

\subsection{Tensors}
\label{ssc:Ten}

\begin{definition}
Let $ (U, U', \scalp{\cdot}{\cdot}_1) $ and $ (V, V', \scalp{\cdot}{\cdot}_2) $ be dual
pairs of vector spaces. The \emph{tensor product} of $ U $ and $ V $ with
respect to the preceding dual pairs is the space of all bilinear functionals
$ U' \times V' \to \RR $.

Observe that all tensor products of two fixed spaces $ U $ and $ V $ (with respect to different
dual pairs) are naturally isomorphic and will be therefore all denoted by $ U \otimes V $.

The elements of $ U \otimes V $ are called \emph{tensors}. For $ u \in U $ and $ v \in V $,
we define the \emph{elementary tensor} $ u \otimes v $ by $ (u \otimes v)(u', v') :=
\scalp{u}{u'}_1 \scalp{v}{v'}_2 $.
\end{definition}

Observe that each tensor is a sum of elementary tensors. Moreover, if $ e_1, \ldots e_m $
is a basis of $ U $ and $ f_1, \ldots, f_n $ is a basis of $ V $, then the elementary
tensors $ e_i \otimes f_j $, $ i = 1, \ldots, m $, $ j = 1, \ldots, n $, form a basis
of $ U \otimes V $.

Furthermore, observe that
for each bilinear map $ \Phi \Colon U \times V \to W $, there exists a unique linear map
$ L_\Phi \Colon U \otimes V \to W $, such that
$ L_\Phi(u \otimes v) = \Phi(u, v) $ for all $ u \in U $ and $ v \in V $. The latter
fact often serves as a definition of the tensor product.

By our definition, however, each tensor $ \phi \in U \otimes V $ is a bilinear map
$ U' \times V' \to \RR $ and can be therefore assigned the linear map
$ L_\phi \Colon U' \otimes V' \to \RR $. Observe that the map $ (\phi, \phi') \mapsto
\scalp{\phi}{\phi'}_\otimes := L_\phi \phi' $ is a non-singular pairing between $ U \otimes V $
and $ U' \otimes V' $ characterized by $ \scalp{u \otimes v}{u' \otimes v'}_\otimes
= \scalp{u}{u'}_1 \scalp{v}{v'}_2 $. Thus,
$ (U \otimes V, U' \otimes V', \scalp{\cdot}{\cdot}_\otimes) $ is a dual pair of
vector spaces.

Each tensor $ \phi \in U \otimes V $ can also be assigned a linear map $ \tilde L_\phi \Colon
V' \to U $ characterized by $ \scalp{\tilde L_\phi v'}{u'}_1 = \phi(u', v') = \scalp{\phi}{u' \otimes v'} $.
The latter will be identified with the tensor itself, so that we shall simply write $ \phi v' $ for
$ \tilde L_\phi v' $. Thus,
\begin{equation}
\label{eq:ActTensScalp}
 \scalp{\phi v'}{u'} = \scalp{\phi}{u' \otimes v'} \, .
\end{equation}
Observe also that
\begin{equation}
\label{eq:ActPureTens}
 (u \otimes v) v' = \scalp{v}{v'} u \, .
\end{equation}
If $ e_1, \ldots, e_m $ is a basis of $ U $, $ f_1, \ldots f_n $
a basis of $ V $ and $ f'_1, \ldots, f'_n $ the underlying dual basis of $ V' $,
the tensor $ \sum_{i,j} a_{ij} \, e_i \otimes f_j $ is identified with the linear
map with matrix $ [a_{ij}]_{i,j} $ with respect to the bases $ f'_1, \ldots, f'_n $
and $ e_1, \ldots, e_m $. If $ V $ is an Euclidean space, $ u \otimes v $ is identified
with $ u v^\Trans $.

The tensor product $ \RR \otimes V $ is naturally isomorphic to $ V $; the latter
two will therefore be identified. Next, the tensor products $ (U \otimes V) \otimes W $
and $ U \otimes (V \otimes W) $ are also naturally isomorphic and will be denoted
by $ U \otimes V \otimes W $. Similarly, we write $ V_1 \otimes V_2 \otimes \cdots \otimes V_r $
and denote the elementary tensors by $ v_1 \otimes v_2 \otimes \cdots \otimes v_r $.
We shall also denote
$ V^{\otimes r} = \underbrace{v \otimes \cdots \otimes v}_r $ and
$ v^{\otimes r} = \underbrace{v \otimes \cdots \otimes v}_r $.

\begin{definition}
\label{df:InjProj}
Let $ (U, U', \scalp{\cdot}{\cdot}_1, |\cdot|_1, |\cdot|'_1) $ and
$ (V, V', \scalp{\cdot}{\cdot}_2, |\cdot|_2, |\cdot|'_2) $ be dual pairs of normed spaces.
The \emph{injective norm} on $ U \otimes V $ is defined by
$ |\phi|_\vee := \sup \{ |\scalp{\phi}{u' \otimes v'}_\otimes| \sth |u'|'_1 \le 1, |v'|'_2 \le 1 \} $
The \emph{projective norm} on $ U' \otimes V' $ is the norm dual to the injective norm
and will be denoted by $ |\cdot|'_\wedge $, i.~e.,
$ |\phi'|'_\wedge := \sup \{ |\scalp{\phi}{\phi'}_\otimes| \sth |\phi|_\vee \le 1 \} $.
\end{definition}

For details on the projective and the injective norm, the reader is referred to Defant
and Floret~\citep{TenNrm}. Notice that the authors use a different (in fact, more common)
definition of the projective norm, but our definition is equivalent: see Section~3.2,
Equation~$(1)$ ibidem.

The injective and the projective norm are both \emph{cross norms}, i.~e.,
$ |u \otimes v|_\vee = |u|_1 |v|_2 $ and $ |u' \otimes v'|'_\wedge = |u'|'_1 |v'|'_2 $.
Next, observe that the natural isomorphism between $ (U \otimes V) \otimes W $
and $ U \otimes (V \otimes W) $ is an isometry if we take the injective norm
in all tensor products. Similarly, the natural isomorphism between $ (U' \otimes V') \otimes W' $
and $ U' \otimes (V' \otimes W') $ is an isometry if we take the projective norm
in all cases. Therefore, there is an unambiguous injective norm on $ U \otimes V \otimes W $
and an unambiguous projective norm on $ U' \otimes V' \otimes W' $.

Recall that, by the Hahn--Banach theorem,
$ |\phi|_\vee = \sup \{ |\scalp{\phi}{\phi'}_\otimes| \sth |\phi'|_\wedge \le 1 \} $.
In other words, if $ U' \otimes V' $ is endowed with the projective norm, the injective
norm of $ \phi $ matches the operator norm of the underlying linear functional $ L_\phi $.
In addition, observe that it also matches the operator norm of the underlying linear map
$ \tilde L_\phi \Colon V' \to U $.

\begin{proposition}
\label{pr:NormTensorOp}
Keeping the notation from Definition~\ref{df:InjProj}, take another dual pair of normed spaces
$ (Z, Z', \scalp{\cdot}{\cdot}_3, |\cdot|_3, |\cdot|'_3) $.
Recalling that each bilinear map $ \Phi \Colon U' \times V' \to Z $ can be assigned a linear map
$ L_\Phi \Colon U' \otimes V' \to Z $, we have:
\[
 \sup \bigl\{ |\Phi(u', v')|_3 \sth |u'|'_1 \le 1, |v'|'_2 \le 1 \bigr\}
 =
 \sup \bigl\{ |L_\Phi \phi'|_3 \sth |\phi'|'_\wedge \le 1 \bigr\} \, .
\]
\end{proposition}

\begin{proof}
Take $ z' \in Z' $ and consider the bilinear map $ \phi_{z^\prime}(u', v')
:= \scalp{\Phi(u', v')}{z'}_3 $, which is in fact a tensor in $ U \otimes V $. Observe that
$ \scalp{\phi_{z^\prime}}{\phi'}_\otimes = L_{\phi_{z^\prime}} \phi' = \scalp{L_\Phi \phi'}{z'}_3 $
for all $ \phi' \in U' \otimes V' $. Now consider its injective norm
$
 |\phi_{z^\prime}|_\vee
 =
 \sup \{ |\scalp{\phi_{z^\prime}}{u' \otimes v'}_\otimes| \sth |u'|'_1 \le 1, |v'|'_2 \le 1 \}
 =
 \sup \{ |\scalp{\phi_{z^\prime}}{\phi'}_\otimes| \sth |\phi'|'_\wedge \le 1 \}
$.
The latter equality can be rewritten as
$
 \sup \{ |\phi_{z^\prime}(u', v')| \sth |u'|'_1 \le 1, |v'|'_2 \le 1 \}
 =
 \sup \{ |L_{\phi_{z^\prime}} \phi'| \sth |\phi'|'_\wedge \le 1 \}
$
or
$
 \sup \{ |\scalp{\Phi(u', v')}{z^\prime}_3| \sth |u'|'_1 \le 1, |v'|'_2 \le 1 \}
 =
 \sup \{ |\scalp{L_\Phi \phi'}{z'}_3| \sth |\phi'|'_\wedge \le 1 \}
$.
Therefore,
$
 \sup \{ |\Phi(u', v')|_3 \sth |u'|'_1 \le 1, |v'|'_2 \le 1 \}
 =
 \sup \{ |\scalp{\Phi(u', v')}{z'}_3| \sth |u'|'_1 \le 1, |v'|'_2 \le 1, |z'|'_3 \le 1 \}
 = 
 \sup \{ |\scalp{L_\Phi \phi'}{z'}_3| \sth |\phi'|'_\wedge \le 1, |z'|'_3 \le 1 \}
 =
 \sup \{ |L_\Phi \phi'|_3 \sth |\phi'|'_\wedge \le 1 \}
$.
This completes the proof.
\end{proof}

\begin{definition}
Let $ (V, V', \scalp{\cdot}{\cdot}) $ be a dual pair of vector spaces. A tensor
$ \phi \in V^{\otimes r} $ is \emph{symmetric} if $
\scalp{\phi}{v'_1 \otimes \cdots \otimes v'_r}
=
\scalp{\phi}{v'_{\pi(1)} \otimes \cdots \otimes v'_{\pi(r)}}
$ for all permutations $ \pi $ of indices $ 1, 2, \ldots, r $.
\end{definition}

\begin{proposition}[Banach~\citep{Ban}; Bochnak and Siciak~\citep{BSi}]
\label{pr:InjSym}
Let $ (V, V', \scalp{\cdot}{\cdot}, |\cdot|, |\cdot|') $ be a dual pair of normed spaces.
Take $ r \in \NN $. If $ \phi \in V^{\otimes r} $ is a symmetric tensor, then
$ |\phi|_\vee = $\newline $ \sup_{|v'|' \le 1} \bigl| \scalp{\phi}{(v')^{\otimes r}} \bigr| $.\qed
\end{proposition}

In the sequel (and in the main part of the paper), we omit the indices at the pairings and the norms
except for $ |\cdot|_\vee $ and $ |\cdot|_\wedge $, the injective and the projective norm.

\subsection{Derivatives as tensors}
\label{ssc:Der}

Throughout this subsection, let $ D \subseteq V $ be an open set.

If $ H \Colon D \to U $ is differentiable at $ x \in D $,
denote its underlying derivative by $ \nabla H(x) $. This is a linear map $ V \to U $.
If $ (V, V', \scalp{\cdot}{\cdot}) $ is a dual pair of vector spaces, $ \nabla H(x) $
can be identified with a tensor in $ U \otimes V' $.

From now on, assume that $ U = \RR^m $ and $ V = \RR^n $ are Euclidean spaces with
standard bases $ e_1, \ldots, e_m $ and $ f_1, \ldots, f_n $. Writing
$ H(x) = \sum_i h_i(x) \, e_i $, we have
$ \nabla H(x) = \sum_i \sum_j \frac{\partial h_i}{\partial x_j}(x) \, e_i \otimes f_j $.

Recall that $ \RR \otimes V $ can be identified with $ V $. Thus, if $ g \Colon D \to \RR $
is differentiable at $ x $, we have
$ \nabla g(x) = \sum_j \frac{\partial g}{\partial x_j}(x) \, f_j $. Thus, in this case,
$ \nabla $ denotes the gradient, as usual. Observe also that for a fixed
$ u \in U $, we have $ \nabla(u g) = u \otimes \nabla g $.

If $ g \Colon D \to U $ is $ r $ times differentiable at $ x $, the $ r $-fold derivative
$ \nabla^r g(x) = \underbrace{\nabla \cdots \nabla}_r g(x) $ at $ x $ can be regarded as an
element of $ U \otimes V^{\otimes r} $ and we have
\[
 \nabla^r g(x)
 =
 \sum_{1 \le j_1, j_2, \ldots, j_r \le n}
   \frac{\partial^r g(x)}{\partial x_{j_1} \, \partial x_{j_2} \cdots \partial x_{j_r}} \otimes
   f_{j_1} \otimes f_{j_2} \otimes \cdots \otimes f_{j_r}
 \, .
\]
In addition, for $ U = \RR $ and $ h(x) = \scalp{\nabla^r g(x)}{v_1 \otimes \cdots \otimes v_r} $,
we have
\[
 \scalp{\nabla h(x)}{v_{r+1}} = \scalp{\nabla^{r+1} g(x)}{v_1 \otimes \cdots \otimes v_r \otimes v_{r+1}} \, .
\]

\subsection{Vector measures}

Throughout this subsection, let $ (\Xi, \scrpt X) $ and $ (\Upsilon, \scrpt Y) $ denote measurable spaces,
i.~e., $ \scrpt X $ is a $ \sigma $-algebra on the set $ \Xi $ and $ \scrpt Y $ is a
$ \sigma $-algebra on $ \Upsilon $.

Let $ \lambda $ be a positive measure on $ (\Xi, \scrpt X) $ and let $ V $ be a vector
space. By $ L^1(\lambda, V) $, denote the space of all Borel measurable maps $ H \Colon \Xi \to V $ with
$ \int |H| \,\dl \lambda < \infty $, where $ |\cdot| $ is a norm on $ V $. Since vector spaces
are assumed to be finite-dimensional, the definition is independent of the norm.

\begin{definition}
Let $ (V, V', \scalp{\cdot}{\cdot}) $ be a dual pair of vector spaces. For each $ H \in L^1(\lambda, V) $ and each
$ v' \in V' $, the function $ \xi \mapsto \scalp{H(\xi)}{v'} $ is $ \lambda $-integrable. Moreover, there
exists a unique vector $ v \in V $, such that $ \scalp{v}{v'} = \int_\Xi \scalp{H(\xi)}{v'} \, \lambda(\dl \xi) $
for all $ v' \in V' $.
The vector $ v $ is called the \emph{integral} of $ H $ with respect to $ \nu $ and is denoted by
$ \int_\Xi H \,\dl \nu $ or $ \int_\Xi H(\xi) \, \nu(\dl \xi) $. The integral is independent of
the choice of $ V' $.
\end{definition}

\begin{remark}
In infinite-dimensional spaces, the concept of integral is not so evident and there are several ones:
see Chapter~2 of Diestel and Uhl~\citep{DiUVM}.
\end{remark}

\begin{remark}
\label{rk:LinIntVectMeas}
For a linear map $ L $, we have $ \int_\Xi L H \, \dl \nu = L \int_\Xi H \,\dl \nu $.
\end{remark}

\begin{definition}
A $ V $-valued measure on $ (\Xi, \scrpt X) $ is a map $ \nu \Colon \scrpt X \to V $, such that
for each sequence $ A_1, A_2, \ldots $ of disjoint measurable sets, we have
$ \nu \bigl( \bigcup_{k=1}^\infty A_k \bigr) = \sum_{k=1}^\infty \nu(A_k) $. A sum of
a $ V $-valued series $ \sum_{k=1}^\infty v_k $ is assumed to exist if $ \sum_{k=1}^\infty
|v_k| < \infty $; again, $ |\cdot| $ can be an arbitrary norm on $ V $.
\end{definition}

If $ \lambda $ is a positive measure on $ (\Xi, \scrpt X) $ and $ H \in L^1(\lambda, V) $,
the map $ \nu := H \cdot \lambda \Colon \scrpt X \to V $ defined by $ \nu(A) := \int_A H \,\dl \lambda $
is a $ V $-valued measure. Conversely, for each $ \sigma $-finite positive measure $ \lambda $ and
each $ \lambda $-continuous vector measure $ \nu $ (i.~e., $ \nu(A) = 0 $
for each $ A $ with $ \lambda(A) = 0 $), there exists a map $ H \in L^1(\lambda, V) $, such that
$ \nu = H \cdot \lambda $. Since $ V $ is assumed to be finite-dimensional, this can be deduced from the
classical Radon--Nikod\'ym theorem. In general, this is not true (see Chapter~3 of Diestel and
Uhl~\citep{DiUVM}).

Observe that for any finite collection $ \nu_1, \ldots, \nu_n $ of vector measures on
the same measurable space and with values in $ V_1, V_2, \ldots, V_n $, respectively,
there exists a finite positive measure $ \lambda $, such that all measures $ \mu_k $
are $ \lambda $-continuous (one can take $ \lambda = \sum_{k=1}^n |\mu_k| $).
Thus, there also exist functions $ H_k \in L^1(\lambda, V_k) $, such that
$ \nu_k = H_k \cdot \lambda $ for all $ k = 1, \ldots, n $.

\begin{definition}
Let $ (V, |\cdot|) $ be a normed space. The \emph{total variation} of a $ V $-valued vector measure
$ \nu $ is defined as $ |\nu|(A) := \sup \sum_{k=1}^n |\nu(A_k)| $, where the supremum runs over
all finite measurable partitions $ A_1, A_2, \ldots, A_n $ of the set $ A $.
\end{definition}

The total variation of a vector measure is a positive measure. As $ V $ is assumed to be
finite-dimensional, it is finite. This can be deduced from the corresponding properties
of real measures: see Theorems~6.2 and 6.4 of Rudin~\citep{Rud2}.

If $ \nu = H \cdot \lambda $, where $ H $ is a positive measure, we have $ |\nu| = |H| \cdot \lambda $:
see Theorem~4 in Section~2 of Chapter~2 of Diestel and Uhl~\citep{DiUVM}.

If $ L \Colon V \to W $ is a linear map and $ \nu $ is a $ V $-valued vector measure, we
define a new vector measure $ L \nu $ by $ (L \nu)(A) := L \nu(A) $. In particular, if
$ (V, V', \scalp{\cdot}{\cdot}) $ is a dual pair of vector spaces, we define a new real
measure $ \scalp{\nu}{v'} $.

\begin{definition}
\label{df:IntTen}
Let $ (U, U', \scalp{\cdot}{\cdot}) $ and $ (V, V', \scalp{\cdot}{\cdot}) $ be dual pairs
of vector spaces and let $ \nu $ be a $ V $-valued vector measure on $ (\Xi, \scrpt X) $.
By $ L^1(\nu, U) $, we denote the space of all Borel measurable functions $ G \Colon \Xi \to U $,
such that $ \int_\Xi |\scalp{G}{u'}| \,\dl |\scalp{\nu}{v'}| < \infty $ for all $ u' \in U' $
and all $ v' \in V' $.

For $ G \in L^1(\nu, U) $, define the integral $ \int_\Xi G \otimes \dl \nu = \int_\Xi G(\xi) \otimes \nu(\dl \xi) $
as the unique tensor $ \phi \in U \otimes V $ which satisfies $ \scalp{\phi}{u' \otimes v'}
= \int_\Xi \scalp{G}{u'} \, \dl \scalp{\nu}{v'} $ for all $ u' \in U' $ and all $ v' \in V' $.
Observe that the definitions of $ L^1(\nu, U) $ and $ \int_\Xi G \otimes \dl \nu $
are independent of the choice of $ U' $ and $ V' $.

This allows us to define a new $ (U \otimes V) $-valued vector measure $ G \otimes \nu $.

For a bilinear map $ \Phi \Colon U \times V \to W $, define $ \int_\Xi \Phi(G, \dl \nu) :=
\int_\Xi \Phi \bigl( G(\xi), \dl \nu(\xi) \bigr) := L_\Phi \int_\Xi G \otimes \dl \nu $.
\end{definition}

\begin{proposition}
\label{pr:IntBilhlambda}
Let $ G $ and $ \nu $ be as above. If $ \nu = H \cdot \lambda $, where $ \lambda $ is a positive measure,
we have
$
 \int_\Xi \Phi(G, \dl \nu) = \int_\Xi \Phi \bigl( G(\xi), H(\xi) \bigr) \, \lambda(\dl \xi)
$. In particular, $ G \otimes \nu = (G \otimes H) \cdot \lambda $.
\end{proposition}

\begin{proof}
It suffices to prove that $ \int_\Xi G \otimes \dl \nu = \int_\Xi (G \otimes H) \,\dl \lambda $
for all $ G \in L^1(\nu, U) $, where $ U $ is a vector space. Let $ \nu $ be $ V $-valued and let
$ (U, U', \scalp{\cdot}{\cdot}) $ and $ (V, V', \scalp{\cdot}{\cdot}) $ be dual pairs
of vector spaces. Then it suffices to check that $ \int_\Xi \scalp{G}{u'} \, \dl \scalp{\nu}{v'}
= \int_\Xi \scalp{G}{u'} \, \dl \scalp{H}{v'} \,\dl \lambda $ for all $ G \in L^1(\nu, U) $,
$ u \in U' $ and $ v \in V' $. However, the latter is equivalent to the claim that
$ \scalp{\nu}{v'} = \scalp{H}{v'} \cdot \lambda $, which follows from Remark~\ref{rk:LinIntVectMeas}.
\end{proof}

\begin{proposition}
\label{pr:IntVectMeasEst}
Let $ U $, $ V $ and $ W $ be normed spaces. For each $ V $-valued vector measure $ \nu $
on $ (\Xi, \scrpt X) $ and each $ G \in L^1(\nu, U) $, we have
$ |G \otimes \nu|_\vee = |G \otimes \nu|_\wedge = |G| \cdot \nu $.

In addition, take a bilinear map $ \Phi \Colon U \times V \to W $. If $ |\Phi(u, v)| \le a \, |u| \, |v| $
for all $ u \in U $ and all $ v \in V $, we also have $ \bigl| \int_\Xi \Phi(G, \dl \nu) \bigr|
\le a \int_\Xi |G| \,\dl |\nu| $.
\end{proposition}

\begin{proof}
Write $ \nu = H \cdot \lambda $, where $ \lambda $ is a finite positive measure and where
$ H \in L^1(\lambda, V) $. Now observe that, by Proposition~\ref{pr:IntBilhlambda},
$
 |G \otimes \nu|_\vee
 =
 |(G \otimes H) \cdot \lambda|_\vee
 =
 |G \otimes H|_\vee \cdot \lambda
 =
 (|G| \, |H|) \cdot \lambda
 =
 |G| \cdot |\nu|
$. An analogous observation holds for the the projective norm.
To prove the second part, observe that
$
 \bigl| \int_\Xi \Phi(G, \dl \nu) \bigr|
 =
 \bigl| \int_\Xi \Phi \bigl( G(\xi), H(\xi) \bigr) \, \lambda(\dl \xi) \bigr|
 \le
 a \int_\Xi |G| \, |H| \,\dl \lambda
 =
 a \int_\Xi |G| \,\dl |\nu| $.
\end{proof}

Now take measurable spaces $ (\Xi_1, \scrpt X_1), \ldots, (\Xi_r, \scrpt X_r) $.
Let $ \nu_1 $ be a $ V_1 $-valued vector measure on $ (\Xi_1, \scrpt X_1) $. Next, for
each $ k = 1, 2, \ldots, r $ take a \emph{transition kernel} $ \nu_k \Colon \Xi_1 \times \cdots \times \Xi_{k-1}
\times \scrpt X_k \to V_k $, i.~e., assume that the map $ A \mapsto \nu_k(\xi_1, \ldots, \xi_{k-1}, A) $
is a $ V_k $-valued vector measure for all $ \xi_1 \in \Xi_1, \ldots, \xi_{k-1} \in \Xi_{k-1} $,
and that the map $ (\xi_1, \ldots, \xi_{k-1}) \mapsto \nu_k(\xi_1, \ldots, \xi_{k-1}, A) $
is measurable with respect to the product $ \sigma $-algebra $ \scrpt X_1 \otimes \cdots \otimes
\scrpt X_{k-1} $ for all $ A \in \scrpt X_k $. Finally, let $ G \Colon \Xi_1 \times \cdots \times \Xi_r
\to U $ be a product measurable function. Then one can consider the integral
\begin{align*}
 & J :=
 \\
 & \int_{\Xi_1} \cdots \int_{\Xi_{r-1}} \int_{\Xi_r}
 G(\xi_1, \ldots, \xi_r)
   \otimes
 \nu_r(\xi_1, \ldots, \xi_{r-1}, \dl \xi_r)
   \otimes
 \nu_{r-1}(\xi_1, \ldots, \xi_{r-2}, \dl \xi_{r-1})
   \otimes \cdots \otimes
 \nu_1(\dl \xi_1)
 \, ,
\end{align*}
provided than all relevant functions are in the suitable $ L^1 $ spaces.

\begin{definition}
\label{df:Phisuccmeas}
Let $ G, \nu_1, \ldots, \nu_r $ be as before and let $ \Phi \Colon U \times V_1 \times \cdots \times V_r \to W $
be a $ (r+1) $-linear map. There exists a unique linear map
$ L_\Phi \Colon U \otimes V_1 \otimes \cdots \otimes V_r \to W $, such that
$ L_\Phi(u \otimes v_1 \otimes \cdots \otimes v_r) = \Phi(u, v_1, \ldots, v_r) $
for all $ u \in U $, $ v_1 \in V_1 $, \ldots, $ v_r \in V_r $. Define:
\[
 \int_{\Xi_1} \cdots \int_{\Xi_{r-1}} \int_{\Xi_r}
 \Phi \Bigl(
   G(\xi_1, \ldots, \xi_r), \>
   \nu_r(\xi_1, \ldots, \xi_{r-1}, \dl \xi_r), \>
   \nu_{r-1}(\xi_1, \ldots, \xi_{r-2}, \dl \xi_{r-1}), \>
     \ldots, \>
   \nu_1(\dl \xi_1)
 \Bigr)
\]
to be $ L_\Phi J $, provided that $ J $ defined as above exists.
\end{definition}

\begin{remark}
The preceding definition may be ambiguous as it may not be clear which variable is associated
to which space. More strictly, one should write $ \int_{\xi_1 \in \Xi_1} \cdots \int_{\xi_{r-1} \in \Xi_{r-1}}
\int_{\xi_r \in \Xi_r} $. However, in the main part of the present paper, we shall always integrate
over the same space $ \Xi $.
\end{remark}

\begin{definition}
For a $ U $-valued vector measure $ \mu $ on $ (\Xi, \scrpt X) $ and a $ V $-valued
vector measure $ \nu $ on $ (\Upsilon, \scrpt Y) $, define the vector measure
$ \mu \otimes \nu $ on $ (\Xi \times \Upsilon, \scrpt X \otimes \scrpt Y) $ as the unique
$ U \otimes V $-valued measure which satisfies $ (\mu \otimes \nu)(A \times B)
= \mu(A) \otimes \nu(B) $ for all $ A \in \scrpt X $ and all $ B \in \scrpt Y $.
\end{definition}

To see that $ \mu \otimes \nu $ actually exists, we can define it componentwise in terms
of product measures. Observe that in the case $ U = V = \RR $, $ \mu \otimes \nu $ coincides
with the usual product measure. Next, observe that if $ \mu = G \cdot \kappa $ and
$ \nu = H \cdot \lambda $, where $ \kappa $ and $ \lambda $ are positive measures,
we have
\[
 \int_{\Xi \times \Upsilon} \Phi \bigl( F, \dl \mu \otimes \dl \nu \bigr) =
 \int_{\Xi \times \Upsilon} \Phi \bigl( F(\xi, \eta), G(\xi) \otimes h(\eta) \bigr) \,
 \kappa(\dl \xi) \otimes \lambda(\dl \eta)
\]
for all bilinear maps $ \Phi \Colon Z \times (U \otimes V) \to W $ and all maps
$ F \in L^1(\mu \otimes \nu, W) $. In particular, one can briefly write
$ (\mu \otimes \nu)(\dl \xi \otimes \dl \eta) = G(\xi) \otimes H(\eta) \>
\kappa(\dl \xi) \otimes \lambda(\dl \eta) $. In other words, letting
$ \breve G(\xi, \eta) := G(\xi) $ and $ \breve H(\xi, \eta) := H(\eta) $, we have
$ \mu \otimes \nu = (\breve G \otimes \breve H) \cdot (\kappa \otimes \lambda) $.

\begin{proposition}
For any two vector measures $ \mu $ and $ \nu $ with values in normed spaces, we have
$ |\mu \otimes \nu|_\vee = |\mu \otimes \nu|_\wedge = |\mu| \otimes |\nu| $.
\end{proposition}

\begin{proof}
Write $ \mu = G \cdot \kappa $ and $ \nu = H \cdot \lambda $, where $ \kappa $ and $ \lambda $
are positive measures and $ G $ and $ H $ are in the suitable $ L^1 $ spaces.
Letting $ \breve G $ and $ \breve H $ be as above, observe that
$
 |\mu \otimes \nu|_\vee
 =
 |\breve G \otimes \breve H|_\vee \cdot (\kappa \otimes \lambda)
 =
 |\breve G| \, |\breve H| \cdot (\kappa \otimes \lambda)
 =
 (|G| \cdot \kappa) \otimes (|H| \cdot \lambda)
 =
 |\mu| \otimes |\nu|
$
and similarly for the projective norm.
\end{proof}

\bigskip

\ACKNO{%
  The author is grateful to Robert Gaunt and Ywik Swan for fruitful discussions.%
}
\bibliographystyle{siam}
\bibliography{LipCLT}    %!S

\end{document}